\documentclass[11pt]{amsart}
\usepackage{amsopn}
\usepackage{amssymb, amscd}
\usepackage{multirow}

\newcommand{\nc}{\newcommand}

\nc{\fg}{\mathfrak{f} } \nc{\vg}{\mathfrak{v} } \nc{\wg}{\mathfrak{w} }
\nc{\zg}{\mathfrak{z} } \nc{\ngo}{\mathfrak{n} } \nc{\kg}{\mathfrak{k} }
\nc{\mg}{\mathfrak{m} } \nc{\bg}{\mathfrak{b} } \nc{\ggo}{\mathfrak{g} }
\nc{\ggob}{\overline{\mathfrak{g}} } \nc{\sog}{\mathfrak{so} }
\nc{\sug}{\mathfrak{su} } \nc{\spg}{\mathfrak{sp} } \nc{\slg}{\mathfrak{sl} }
\nc{\glg}{\mathfrak{gl} } \nc{\cg}{\mathfrak{c} } \nc{\rg}{\mathfrak{r} }
\nc{\hg}{\mathfrak{h} } \nc{\tg}{\mathfrak{t} } \nc{\ug}{\mathfrak{u} }
\nc{\dg}{\mathfrak{d} } \nc{\ag}{\mathfrak{a} } \nc{\pg}{\mathfrak{p} }
\nc{\sg}{\mathfrak{s} } \nc{\affg}{\mathfrak{aff} }

\nc{\pca}{\mathcal{P}} \nc{\nca}{\mathcal{N}} \nc{\lca}{\mathcal{L}}
\nc{\oca}{\mathcal{O}} \nc{\mca}{\mathcal{M}} \nc{\tca}{\mathcal{T}}
\nc{\aca}{\mathcal{A}} \nc{\cca}{\mathcal{C}} \nc{\gca}{\mathcal{G}}
\nc{\sca}{\mathcal{S}} \nc{\hca}{\mathcal{H}} \nc{\bca}{\mathcal{B}}
\nc{\dca}{\mathcal{D}} \nc{\val}{\operatorname{val}}

\nc{\vp}{\varphi} \nc{\ddt}{\frac{d}{dt}} \nc{\dds}{\frac{d}{ds}}
\nc{\dpar}{\frac{\partial}{\partial t}} \nc{\im}{\mathtt{i}}

\nc{\SO}{\mathrm{SO}} \nc{\Spe}{\mathrm{Sp}} \nc{\Sl}{\mathrm{SL}}
\nc{\SU}{\mathrm{SU}} \nc{\Or}{\mathrm{O}} \nc{\U}{\mathrm{U}} \nc{\Gl}{\mathrm{GL}}
\nc{\Se}{\mathrm{S}} \nc{\Cl}{\mathrm{Cl}} \nc{\Spein}{\mathrm{Spin}}
\nc{\Pin}{\mathrm{Pin}} \nc{\G}{\mathrm{GL}_n(\RR)} \nc{\g}{\mathfrak{gl}_n(\RR)}

\nc{\RR}{{\Bbb R}} \nc{\HH}{{\Bbb H}} \nc{\CC}{{\Bbb C}} \nc{\ZZ}{{\Bbb Z}}
\nc{\FF}{{\Bbb F}} \nc{\NN}{{\Bbb N}} \nc{\QQ}{{\Bbb Q}} \nc{\PP}{{\Bbb P}}

\nc{\vs}{\vspace{.2cm}} \nc{\vsp}{\vspace{1cm}} \nc{\ip}{\langle\cdot,\cdot\rangle}
\nc{\ipp}{(\cdot,\cdot)} \nc{\la}{\langle} \nc{\ra}{\rangle} \nc{\unm}{\tfrac{1}{2}}
\nc{\unc}{\tfrac{1}{4}} \nc{\und}{\tfrac{1}{16}} \nc{\no}{\vs\noindent}
\nc{\lam}{\Lambda^2(\RR^n)^*\otimes\RR^n} \nc{\tangz}{{\rm T}^{\rm Zar}}
\nc{\nor}{{\sf n}}  \nc{\mum}{/\!\!/} \nc{\kir}{/\!\!/\!\!/}
\nc{\Ri}{\tfrac{4\Ric_{\mu}}{||\mu||^2}} \nc{\ds}{\displaystyle}
\nc{\ben}{\begin{enumerate}} \nc{\een}{\end{enumerate}} \nc{\f}{\frac}
\nc{\lb}{[\cdot,\cdot]} \nc{\isn}{\tfrac{1}{||v||^2}}
\nc{\gkp}{(\ggo=\kg\oplus\pg,\ip)} \nc{\ukh}{(\ug=\kg\oplus\hg,\ip)}
\nc{\tgkp}{(\tilde{\ggo}=\kg\oplus\pg,\ip)}
\nc{\wt}{\widetilde} \nc{\mm}{M}

\nc{\Hess}{\operatorname{Hess}} \nc{\ad}{\operatorname{ad}}
\nc{\Ad}{\operatorname{Ad}} \nc{\rank}{\operatorname{rank}}
\nc{\Irr}{\operatorname{Irr}} \nc{\End}{\operatorname{End}}
\nc{\Aut}{\operatorname{Aut}} \nc{\Inn}{\operatorname{Inn}}
\nc{\Der}{\operatorname{Der}} \nc{\Ker}{\operatorname{Ker}}
\nc{\Iso}{\operatorname{I}} \nc{\Diff}{\operatorname{Diff}}
\nc{\Lie}{\operatorname{L}} \nc{\tr}{\operatorname{tr}} \nc{\dif}{\operatorname{d}}
\nc{\sen}{\operatorname{sen}} \nc{\modu}{\operatorname{mod}}
\nc{\CRic}{\operatorname{PP}} \nc{\Cric}{\operatorname{P}} \nc{\Ricci}{\operatorname{Ric}}
\nc{\sym}{\operatorname{sym}} \nc{\symac}{\operatorname{sym^{ac}}}
\nc{\symc}{\operatorname{sym^{c}}} \nc{\scalar}{\operatorname{sc}}
\nc{\grad}{\operatorname{grad}} \nc{\ricci}{\operatorname{Rc}}
\nc{\Nor}{\operatorname{Norm}}  \nc{\ricc}{\operatorname{Rc^{c}}}
\nc{\Ricc}{\operatorname{Ric^{c}}} \nc{\ricac}{\operatorname{Rc^{ac}}}
\nc{\Ricac}{\operatorname{Ric^{ac}}} \nc{\Riem}{\operatorname{Rm}}
\nc{\riccig}{\operatorname{ric^{\gamma}}} \nc{\Rin}{\operatorname{M}}
\nc{\Le}{\operatorname{L}} \nc{\tang}{\operatorname{T}}
\nc{\level}{\operatorname{level}} \nc{\rad}{\operatorname{r}}
\nc{\abel}{\operatorname{ab}} \nc{\CH}{\operatorname{CH}}
\nc{\mcc}{\operatorname{mcc}} \nc{\Adj}{\operatorname{Adj}}
\nc{\Order}{\operatorname{O}}  \nc{\inj}{\operatorname{inj}} \nc{\proy}{\operatorname{pr}}
\nc{\vol}{\operatorname{vol}} \nc{\Diag}{\operatorname{Diag}}
\nc{\Spec}{\operatorname{Spec}}

\theoremstyle{plain}
\newtheorem{theorem}{Theorem}[section]
\newtheorem{proposition}[theorem]{Proposition}
\newtheorem{corollary}[theorem]{Corollary}
\newtheorem{lemma}[theorem]{Lemma}

\theoremstyle{definition}

\theoremstyle{remark}
\newtheorem{remark}[theorem]{Remark}

\newtheorem{example}[theorem]{Example}

\title[]{On the symplectic curvature flow for locally homogeneous manifolds}

\author{Jorge Lauret} \author{Cynthia Will}

\address{Universidad Nacional de C\'ordoba, FaMAF and CIEM, 5000 C\'ordoba, Argentina}
\email{lauret@famaf.unc.edu.ar} \email{cwill@famaf.unc.edu.ar}

\thanks{This research was partially supported by grants from CONICET, FONCYT and SeCyT (Universidad Nacional de C\'ordoba)}

\dedicatory{Dedicated to the memory of our dear friend Sergio Console.}

\begin{document}

\maketitle

\begin{abstract}
Recently, J. Streets and G. Tian introduced a natural way to evolve an almost-K\"ahler manifold called the symplectic curvature flow, in which the metric, the symplectic structure and the almost-complex structure are all evolving.  We study in this paper different aspects of the flow on locally homogeneous manifolds, including long-time existence, solitons, regularity and convergence.  We develop in detail two large classes of Lie groups, which are relatively simple from a structural point of view but yet geometrically rich and exotic: solvable Lie groups with a codimension one abelian normal subgroup and a construction attached to each left symmetric algebra.  As an application, we exhibit a soliton structure on most of symplectic surfaces which are Lie groups.  A family of ancient solutions which develop a finite time singularity was found; neither their Chern scalar nor their scalar curvature are monotone along the flow and they converge in the pointed sense to a (non-K\"ahler) shrinking soliton solution on the same Lie group.
\end{abstract}

%\tableofcontents

\section{Introduction}\label{intro}

There is a natural way to evolve an almost-K\"ahler manifold $(M,\omega,g,J)$ which has recently been introduced by J. Streets and G. Tian in \cite{StrTn2} and is called the {\it symplectic curvature flow} (or SCF for short):
\begin{equation}\label{SCFeq-intro}
\left\{\begin{array}{l}
\dpar\omega=-2p, \\ \\
\dpar g=-2p^{1,1}(\cdot,J\cdot)-2\ricci^{2,0+0,2},
\end{array}\right.
\end{equation}
where $p$ is the Chern-Ricci form of $(\omega,g)$ and $\ricci$ is the Ricci tensor of $g$.  The equation for the symplectic form is in direct analogy with K\"ahler-Ricci flow, the term $-2p^{1,1}(\cdot,J\cdot)$ in the equation for the metric guarantees that compatibility is preserved and the term $-2\ricci^{2,0+0,2}$, being geometrically very natural, yields to the (weak) parabolicity of the flow.  The evolution of $J$ follows from the formula $\omega=g(J\cdot,\cdot)$ (see \eqref{evJ}).  Unlike the anti-complexified Ricci flow (see \cite{LeWng}), where $\omega$ remains fixed in time, and unlike some flows for hermitian manifolds studied in the literature like hermitian curvature flow (see \cite{StrTn1}), pluriclosed flow (see \cite{StrTn3}) or Chern-Ricci flow (see \cite{TstWnk}), in which $J$ is fixed along the flow, in SCF the three structures are indeed evolving.  This certainly makes very difficult the study of any basic property of the flow.  A flow unifying SCF and pluriclosed flow is studied in \cite{Dai} and a result on stability of K\"ahler-Einstein structures is given in \cite{Smt}.

Our aim in this paper is to explore some aspects of the SCF on the class of locally homogeneous almost-K\"ahler manifolds, in order to exemplify and provide some evidence for eventual conjectures in the general case (we refer to \cite{Pk,SCF,Frn} for further work on homogeneous SCF).   More precisely, we are interested in the SCF evolution of compact almost-K\"ahler manifolds $(M,\omega,g)$ whose universal cover is a Lie group $G$ and such that if $\pi:G\longrightarrow M$ is the covering map, then $\pi^*\omega$ and $\pi^*g$ are left-invariant (e.g.\ invariant structures on solvmanifolds and nilmanifolds).  A solution on $M$ is therefore obtained by pulling down the corresponding solution on the Lie group $G$, which by diffeomorphism invariance stays left-invariant and so equation \eqref{SCFeq-intro} becomes an ODE for a compatible pair $(\omega(t),g(t))$, where $\omega(t)$ is a closed non-degenerate $2$-form on the Lie algebra $\ggo$ of $G$ and $g(t)$ is an inner product on $\ggo$ for all $t$.  Notice that short-time existence (forward and backward) and uniqueness of the solutions are therefore guaranteed, say on a maximal interval of time $(T_-,T_+)$ containing $0$, $T_\pm\in\RR\cup\{\pm\infty\}$.  We therefore study, more in general, left-invariant solutions on Lie groups which may or may not admit a cocompact discrete subgroup.

An almost-K\"ahler manifold $(M,\omega,g)$ will flow self-similarly along the SCF, in the sense that
$$
(\omega(t),g(t))=(c_t\vp(t)^*\omega,c_t\vp(t)^*g), \qquad \mbox{for some}\; c_t>0, \quad \vp(t)\in\Diff(M),
$$
if and only if the Chern-Ricci form and Ricci tensor of $(\omega,g)$ satisfy
$$
\left\{
\begin{array}{l}
p=c\omega+\lca_{X}\omega, \\ \\
p^{1,1}(\cdot,J\cdot)+\ricci^{2,0+0,2}=cg+\lca_{X}g,
\end{array}\right. \qquad \mbox{for some}\; c\in\RR, \quad X\in\chi(M)\; \mbox{(complete)}.
$$
In analogy to the terminology used in Ricci flow theory, we call such structure $(\omega,g)$ a
{\it SCF-soliton} and we say it is {\it expanding}, {\it steady} or {\it shrinking}, if $c<0$, $c=0$ or $c>0$, respectively.  The following natural open questions were our main motivation: Does every symplectic Lie group $(G,\omega)$ admit a compatible metric $g$ such that $(\omega,g)$ is a SCF-soliton?  Is a SCF-soliton structure unique up to equivalence and scaling?  Are all nonflat SCF-solitons on Lie groups steady or expanding?  Note that the last question is related to the long-time existence of SCF solutions.

The following two large classes of Lie groups have been studied in detail.  We believe that some of the results obtained in the present paper might also be useful in other problems on almost-K\"ahler geometry, specially those involving Chern-Ricci or Ricci curvature.

\subsection{Almost abelian solvmanifolds}\label{aa-intro}
In Section \ref{muA-sec}, we attach to each $(2n-1)\times(2n-1)$-matrix of the form
$$
A=\left[\begin{array}{c|c}
a&v\\\hline &\\
0& \quad A_1 \quad \\ &\\
\end{array}\right],
\qquad a\geq 0, \quad v\in\RR^{2n-2}, \quad A_1\in\spg(n-1,\RR),
$$
a left-invariant almost-K\"ahler structure on a $2n$-dimensional Lie group denoted by $G_A$.  The Lie algebra of $G_A$ has an orthonormal basis $\{ e_1,\dots,e_{2n}\}$ such that $\ngo:=\la e_1,\dots,e_{2n-1}\ra$ is an abelian ideal, $\ad{e_{2n}}|_{\ngo}=A$, and the fixed symplectic form and almost-complex structures are respectively given by
$$
\omega=e^1\wedge e^{2n}+\omega_1, \qquad
J=\left[\begin{array}{c|c|c} 0&0&-1 \\\hline &&\\ 0&\quad J_1\quad &0 \\ &&\\ \hline 1&0&0 \end{array}\right],
$$
where $\{ e^i\}$ denotes the dual basis and $\omega_1=g(J_1\cdot,\cdot)$ is the nondegenerate $2$-form on $\ngo_1:=\la e_2,\dots,e_{2n-1}\ra$ used to define the Lie algebra $\spg(n-1,\RR)$ above.  Any almost-K\"ahler Lie group with a codimension-one abelian normal subgroup is equivalent to $(G_A,\omega,g)$ for some matrix $A$ as above.  The structure is K\"ahler if and only if $v=0$ and $A_1\in\sug(n-1)$.

After giving some criteria for the equivalence between these structures, we compute their Chern-Ricci and Ricci tensors in terms of $A$, which is actually the only datum that is varying in this construction.  We then study the existence, uniqueness and structure of solitons among this class, which turn out to be all expanding if nonflat.

\begin{theorem}\label{SA-sol-intro}
Assume that $v=0$.
\begin{itemize}
\item[(i)] If $A$ is either semisimple or nilpotent, then the symplectic Lie group $(G_A,\omega)$ admits a compatible metric $g$ such that the almost-K\"ahler structure $(\omega,g)$ is a SCF-soliton.  The condition for $(G_A,\omega,g)$ being a SCF-soliton is respectively given by $A$ normal and $$[A,[A,A^t]]=-\left(|[A,A^t]|^2/|A|^2\right)A.$$

\item[(ii)] If $A$ is neither nilpotent nor semisimple, then the Lie group $G_A$ does not admit any (algebraic) SCF-soliton.

\item[(iii)] The SCF evolution is equivalent to the ODE for $A=A(t)$ given by
$$
A'=  -\unm((\tr{A})^2+\tr{S(A)^2})A+\unm[A,[A,A^t]]-\tfrac{\tr{A}}{2}[A,A^t].
$$
\item[(iv)] Any solution $(\omega(t),g(t)$ is immortal (i.e.\ $T_+=\infty$) in this class.

\item[(v)] The quantity $|[A,A^t]|^2/|A|^4$ is strictly decreasing along the flow, unless the solution is a SCF-soliton.

\item[(vi)] Any accumulation point $A_+$ of the set $\{ A(t)/|A(t)|:t\in[0,\infty)\}$ gives rise to a limit soliton $(G_{A_+},\omega,g)$.  If $A_0$ is not nilpotent, then $A_+$ is a normal matrix having the same eigenvalues as $A_0$ up to scaling.
\end{itemize}
\end{theorem}

Exactly five $4$-dimensional symplectic Lie groups admit a lattice, giving rise to the compact symplectic surfaces which are solvmanifolds.  They all admit a SCF-soliton, and since they all have a codimension one abelian normal subgroup, we use the results obtained for almost abelian solvmanifolds described above to study their SCF evolution in Section \ref{surfaces}, including the convergence behavior.

\subsection{LSA construction}\label{lsa-intro}
In order to search for SCF-solitons beyond the solvable case, we considered in Section \ref{weak} a construction attaching to each $n$-dimensional left-symmetric algebra (LSA for short; see \eqref{LSA-def}) an almost-K\"ahler structure on a $2n$-dimensional Lie group (see \cite{Bym,AndSlm,Ovn} for further information on this construction).

We fix an euclidean symplectic vector space  $(\ggo\oplus\ggo,\omega,g,J)$, where $g$ is an inner product making the two copies of the vector space $\ggo$ orthogonal and $J=\left[\begin{smallmatrix} 0&I\\ -I&0\end{smallmatrix}\right]$.  Now for each LSA structure on $\ggo$, define the Lie algebra $\ggo\ltimes_\theta\ggo$ with Lie bracket
$$
[(X,Y),(Z,W)]:=\left( [X,Z]_\ggo, \theta(X)W-\theta(Z)Y \right),
$$
where $[X,Y]_\ggo:=X\cdot Y-Y\cdot X$ is the corresponding Lie bracket on $\ggo$ and $\theta(X):=-L(X)^t$.  Here $L(X)$ denotes LSA left-multiplication by $X\in\ggo$.  The almost-K\"ahler Lie algebra $(\ggo\ltimes_\theta\ggo,\omega,g)$ is therefore completely determined by the LSA structure.  We first prove some criteria on the equivalence between these structures and then compute their Chern-Ricci and Ricci curvature in terms of $\theta$, which is the only datum varying here.

The SCF on this class is equivalent to the ODE for $\theta=\theta(t)$ given by
$$
\theta'(X)=\theta((P_1+S)X) +[\theta(X),P_1^t-S], \qquad\forall X\in\ggo,
$$
where $P=\left[\begin{smallmatrix} P_1&0\\ 0&P_1^t\end{smallmatrix}\right]$ and $\Ricac=\left[\begin{smallmatrix} S&0\\ 0&-S\end{smallmatrix}\right]$ are respectively the Chern-Ricci operator (i.e. $p=\omega(P\cdot,\cdot)$) and the anti-J-invariant part of the Ricci operator $\Ricci$ (i.e. $\Ricac=\unm(\Ricci+J\Ricci J)$).

\begin{theorem}\label{LSA-intro}
Let $G$ denote the $8$-dimensional Lie group with Lie algebra defined as above for the Lie algebra $\ggo=\ug(2)$ with LSA structure coming from the identification $\ggo=\HH$ with the quaternion numbers.
\begin{itemize}
\item[(i)] There is a family of {\it ancient} solutions on $G$ (i.e. $T_-=-\infty$).  Each one of them develops a finite time singularity $T_+<\infty$ (see Example \ref{BF-exa}).

\item[(ii)] The Chern scalar curvature $\tr{p}$ of any of the solutions in part (i) is always positive and, as $t\to T_+$, $\tr{p}\to\infty$ after attaining a global minimum.  The scalar curvature $R$ is always negative, attains a global maximum and $R\to-\infty$, as $t\to T_+$.  In particular, neither $\tr{p}$ nor $R$ are monotone along the flow.

\item[(iii)] $G$ admits a non-K\"ahler shrinking SCF-soliton $(\omega,g)$ (see Example \ref{u(2)}) with Chern-Ricci form, Ricci operator and scalar curvature given respectively by
$$
p=20\omega, \quad \Ricci=\Diag(-100,92,92,92,-244,-52,-52,-52), \quad R=-224.
$$
\item[(iv)] Each ancient solution from part (i) converges in the pointed sense to the shrinking SCF-soliton structure in (iii), and backward, they converge to expanding SCF-solitons on certain solvable Lie groups.
\end{itemize}
\end{theorem}

\begin{remark}
Along the way, we found negative Ricci curvature metrics on the Lie group $G$ in the above theorem which are new in the literature as far as we know (compare with \cite{NklNkn}).
\end{remark}

\subsection{Homogeneous symplectic surfaces}
According to the classification obtained in \cite{Ovn}, there are fourteen $4$-dimensional Lie groups admitting a left-invariant symplectic structure (see Table \ref{n1}).  They are all solvable, some of them are actually continuous families of groups and many of them admit more than one symplectic structure.
We have found in Section \ref{dim4-sec} a (unique) SCF-soliton on each of these symplectic Lie groups, with the exception of only four cases.  For two of them we were able to prove the non-existence of (algebraic) solitons.  The SCF-soliton almost-K\"ahler structures and their respective Chern-Ricci and Ricci operators are given in Table \ref{n3}.  They are all expanding solitons if nonflat and are {\it static} (i.e. $p=c\omega$ and $\ricac=0$) if and only if they are K\"ahler-Einstein.  The last equivalence was proved for any compact static almost-K\"ahler structure of dimension $4$ in \cite[Corollary 9.5]{StrTn2}.

\vs \noindent {\it Acknowledgements.}  We are very grateful to an anonymous referee for helpful observations.

\section{Preliminaries and notation}

Let $\ggo$ be a real vector space.  The following notation will be used for $\ggo$ the tangent space $T_pM$ at a point of a differentiable manifold, as well as for the underlying vector space of a Lie algebra.  We consider an almost-hermitian structure $(\omega,g,J)$ on $\ggo$, that is, a $2$-form $\omega$ and an inner product $g$ such that if
$$
\omega=g(J\cdot,\cdot),
$$
then $J^2=-I$.  The above formula is therefore equivalent to $g=\omega(\cdot,J\cdot)$.

The transposes of a linear map $A:\ggo\longrightarrow\ggo$ with respect to $g$ and $\omega$ are respectively given by
$$
g(A\cdot,\cdot)=g(\cdot,A^t\cdot), \qquad \omega(A\cdot,\cdot)=\omega(\cdot,A^{t_\omega}\cdot), \qquad A^{t_\omega}=-JA^tJ,
$$
and if $p:\ggo\times\ggo\longrightarrow\RR$ is a bilinear map, then their complexified (or $J$-invariant) and anti-complexified (or anti-$J$-invariant) components are defined by
$$
A=A^c+A^{ac}, \qquad A^c:=\unm(A-JAJ), \qquad A^{ac}:=\unm(A+JAJ),
$$
and $p=p^c+p^{ac}$, where
$$
p^c=p^{1,1}:=\unm(p(\cdot,\cdot)+p(J\cdot,J\cdot)),  \qquad p^{ac}=p^{2,0+0,2}:=\unm(p(\cdot,\cdot)-p(J\cdot,J\cdot)).
$$
Let $(M,\omega,g,J)$ be a $2n$-dimensional almost-K\"ahler manifold (i.e. $d\omega=0$).  The {\it Chern connection} is the unique connection $\nabla$ on $M$ which is hermitian (i.e. $\nabla\omega=0$, $\nabla g=0$, $\nabla J=0$) and its torsion satisfies $T^{1,1}=0$.  In terms of the Levi Civita connection $D$ of $g$, the Chern connection is given by
$$
\nabla_XY=D_XY+\unm(D_XJ)JY;
$$
in particular, $\nabla=D$ if and only if $(M,\omega,g,J)$ is K\"ahler.  The {\it Chern-Ricci form} $p=p(\omega,g)$ is defined by
$$
p(X,Y)=\sum_{i=1}^{n} g(R(X,Y)e_i,Je_i) = \sqrt{-1} \sum_{i=1}^{n} g(R(X,Y)Z_i,Z_{\overline{i}}),
$$
where $R(X,Y)=\nabla_{[X,Y]} - [\nabla_X,\nabla_Y]$ is the curvature tensor of $\nabla$ and
$$
\{ e_1,\dots,e_n,Je_1,\dots,Je_n\}
$$
is a local orthonormal frame for $g$ with corresponding local unitary frame
$$
Z_i:=(e_i-\sqrt{-1} Je_i)/\sqrt{2}, \qquad Z_{\overline{i}}:=(e_i+\sqrt{-1} Je_i)/\sqrt{2}.
$$
The Chern-Ricci form is always closed, locally exact and in the K\"ahler case $p$ equals the Ricci form $\ricci(J\cdot,\cdot)$.  By Chern-Weil theory, its cohomology class equals $[p]=2\pi c_1(M,J)$, where $c_1(M,J)\in H^2(M,\RR)$ is the first Chern class.

The Chern-Ricci form $p$ of a left-invariant almost-hermitian structure $(\omega,g,J)$ on a Lie group with Lie algebra $\ggo$ is given by
\begin{equation}\label{CRform}
p(X,Y)=-\unm\tr{J\ad{[X,Y]}} + \unm\tr{\ad{J[X,Y]}}, \qquad\forall X,Y\in\ggo.
\end{equation}
(See \cite[Proposition 4.1]{Vzz2} or \cite{Pk}).  Remarkably, $p$ only depends on $J$.  Since $p$ is exact, there exists a unique $Z\in\ggo$ such that
$$
p(X,Y)=g([X,Y],JZ)=\omega(Z,[X,Y]),
$$
and the {\it Chern-Ricci operator} $P$ defined by $p=\omega(P\cdot,\cdot)$ equals
\begin{equation}\label{PadZ}
P=\ad{Z}+(\ad{Z})^{t_\omega}.
\end{equation}
(See \cite[(2.3)]{Frn}).

\section{Symplectic curvature flow}\label{sec-SCF}

Let $(M,\omega,g,J)$ be an almost-K\"ahler manifold of dimension $2n$, i.e. an almost-hermitian manifold such that $d\omega=0$.  With K\"ahler-Ricci flow as a motivation, it is natural to evolve the symplectic structure $\omega$ in the direction of the Chern-Ricci form $p$, but since in general $p\ne p^c$, one is forced to flow the metric $g$ as well in order to preserve compatibility.  The following evolution equation for a one-parameter family $(\omega(t),g(t))$ of almost-K\"ahler structures has recently been introduced by Streets-Tian in \cite{StrTn2} and is called the {\it symplectic curvature flow} (or SCF for short):
\begin{equation}\label{SCFeq}
\left\{\begin{array}{l}
\dpar\omega=-2p, \\ \\
\dpar g=-2p^c(\cdot,J\cdot)-2\ricac,
\end{array}\right.
\end{equation}
where $p$ is the Chern-Ricci form of $(\omega,g)$ and $\ricci$ is the Ricci tensor of $g$.  SCF-solutions preserve the compatibility and the almost-K\"ahler condition (recall that $dp=0$).  The almost-complex structure evolves as follows:
\begin{equation}\label{evJ}
\dpar J=-2JP^{ac}-2J\Ricac = -2JP^{ac} + [\Ricci,J],
\end{equation}
where $\Ricci$ denotes the Ricci operator of the metric $g$ (i.e. $\ricci=g(\Ricci\cdot,\cdot)$) and $P$ the Chern-Ricci operator (i.e. $p=\omega(P\cdot,\cdot)$).  We note that if $J_0$ is integrable, i.e. $(\omega_0,g_0)$ K\"ahler, then $J=J_0$, $\ricac=0$ and $p^c(\cdot,J\cdot)=\ricci$ for all $t$ and so SCF becomes precisely the K\"ahler-Ricci flow for $g(t)$.

\subsection{SCF on Lie groups}\label{sec-LG}
Our aim in this paper is to study the SCF evolution of compact almost-K\"ahler manifolds $(M,\omega,g)$ whose universal cover is a Lie group $G$ and such that if $\pi:G\longrightarrow M$ is the covering map, then $\pi^*\omega$ and $\pi^*g$ are left-invariant.  This is in particular the case of invariant structures on a quotient $M=G/\Gamma$, where $\Gamma$ is a cocompact discrete subgroup of $G$ (e.g. solvmanifolds and nilmanifolds).  A solution on $M$ is therefore obtained by pulling down the corresponding solution on the Lie group $G$, which by diffeomorphism invariance stays left-invariant and so it can be studied on the Lie algebra or infinitesimal level as an ODE.

Any almost-K\"ahler structure on a Lie group with Lie algebra $\ggo$ which is left-invariant is determined by a pair $(\omega,g)$, where $\omega$ is a non-degenerate $2$-form on the Lie algebra $\ggo$ that is {\it closed}, i.e.
\begin{equation}\label{closed}
\omega([X,Y],Z) + \omega([Y,Z],X) + \omega([Z,X],Y) =0, \qquad\forall X,Y,Z\in\ggo,
\end{equation}
and $g$ is an inner product on the underlying vector space $\ggo$ {\it compatible} with $\omega$ (i.e. if $\omega=g(J\cdot,\cdot)$, then $J^2=-I$).  Two almost-K\"ahler structures $(\ggo_1,\omega_1,g_1)$ and $(\ggo_2,\omega_2,g_2)$ are called {\it equivalent} if there is a Lie algebra isomorphism $\vp:\ggo_1\longrightarrow\ggo_2$ such that $\omega_1=\vp^*\omega_2$ and $g_1=\vp^*g_2$.

Since all the tensors involved are determined by their value at the identity of the group, the SCF equation \eqref{SCFeq} on $M$, or on the covering Lie group $G$, becomes an ODE system of the form
\begin{equation}\label{eqLG}
\left\{\begin{array}{l}
\ddt\omega=-2p, \\ \\
\ddt g=-2p^c(\cdot,J\cdot)-2\ricac,
\end{array}\right.
\end{equation}
where $p=p(\omega,g)\in\Lambda^2\ggo^*$ and $\ricac=\ricac(\omega,g)\in\sca^2\ggo^*$.  Thus short-time existence (forward and backward) and uniqueness of the solutions are always guaranteed.

Given a left-invariant almost-hermitian structure $(\omega_0,g_0)$ on a simply connected Lie group $G$, one has that
\begin{equation}\label{equiv}
(\omega,g)=h^*(\omega_0,g_0):=\left(\omega_0(h\cdot,h\cdot),g_0(h\cdot,h\cdot)\right),
\end{equation}
is also almost-hermitian for any $h\in\Gl(\ggo)$, and conversely, any almost-hermitian structure on $\ggo$ is of this form.  Moreover, the corresponding Lie group isomorphism
$$
\widetilde{h}:(G,\omega,g)\longrightarrow (G_\mu,\omega_0,g_0), \qquad\mbox{where}\qquad \mu=h\cdot\lb:=h[h^{-1}\cdot,h^{-1}\cdot],
$$
is an equivalence of almost-hermitian manifolds.  Here $\lb$ denotes the Lie bracket of the Lie algebra $\ggo$ and so $\mu$ defines a new Lie algebra (isomorphic to $(\ggo,\lb)$) with same underlying vector space $\ggo$.  We denote by $G_\mu$ the simply connected Lie group with
Lie algebra $(\ggo,\mu)$.

In this way, if we fix a compatible pair $(\omega_0,g_0)$ on a vector space $\ggo$ of dimension $2n$, then each left-invariant almost-hermitian structure on each $2n$-dimensional simply connected Lie group can be identified with a point in the variety $\lca$ of $2n$-dimensional Lie algebras defined by
$$
\lca:=\{\mu\in\Lambda^2\ggo^*\otimes\ggo:\mu\;\mbox{satisfies the Jacobi condition}\}.
$$
We denote by $\Spe(\omega_0)$ the subgroup isomorphic to $\Spe(n,\RR)$ of $\Gl(\ggo)$ ($\simeq\Gl_{2n}(\RR)$) given by those elements preserving $\omega_0$ (i.e. $\vp^*\omega_0=\omega_0$) and by $\spg(\omega_0)$ its Lie algebra, which is isomorphic to $\spg(n,\RR)$ and given by the maps $A\in\glg(\ggo)$ such that $A^tJ_0+J_0A=0$.  If
$$
\U(\omega_0,g_0):=\Spe(\omega_0)\cap\Or(g_0),
$$
where $\Or(g_0)$ denotes the subgroup of orthogonal maps (i.e. $\vp^*g_0=g_0$), then $\U(\omega_0,g_0)$ is isomorphic to the unitary group $\U(n)$.  Recall that the map $h$ in \eqref{equiv} is unique only up to left-multiplication by elements in $\U(\omega_0,g_0)$.

Note that $\Gl(\ggo)$-orbits in $\lca$ are precisely Lie isomorphism classes.  We are interested in this paper in the algebraic subset $\lca(\omega_0)\subset\lca$ of those Lie brackets for which the fixed $2$-form $\omega_0$ is closed (see \eqref{closed}), i.e. on those points which are almost-K\"ahler.

Recall that two symplectic Lie algebras $(\ggo_1,\omega_1)$ and $(\ggo_2,\omega_2)$  are said to be {\it isomorphic} if $\omega_1=\vp^*\omega_2$ for some Lie algebra isomorphism $\vp:\ggo_1\longrightarrow\ggo_2$.  Therefore, from the varying Lie brackets viewpoint, $\Spe(\omega_0)$-orbits in $\lca(\omega_0)$ are precisely the isomorphism classes of symplectic Lie algebras
$$
\left\{(\ggo,\mu,\omega_0):\mu\in\lca(\omega_0)\right\}.
$$
On the other hand, by \eqref{equiv}, $\U(\omega_0,g_0)$-orbits in $\lca(\omega_0)$ are the equivalence classes of the almost-K\"ahler structures
$$
\left\{(\ggo,\mu,\omega_0,g_0):\mu\in\lca(\omega_0)\right\}.
$$
It also follows that, given $\mu\in\lca(\omega_0)$, the orbit $\Spe(\omega_0)\cdot\mu$ also parameterizes the set of all left-invariant metrics on $G_\mu$ which are compatible with $\omega_0$.

\begin{example}\label{kodaira}
If $\omega_0=e^1\wedge e^{2n}+\dots+e^{n}\wedge e^{n+1}$ and the only nonzero bracket of $\mu_0\in\lca$ is $\mu_0(e_1,e_2)=e_3$, then $\mu_0\in\lca(\omega_0)$ and is isomorphic to $\hg_3\oplus\RR^{2n-3}$ as a Lie algebra, where $\hg_3$ denotes the $3$-dimensional Heisenberg algebra.  As an almost-K\"ahler structure, $(G_{\mu_0},\omega_0,g_0)$ is equivalent to $(H_3\times\RR)\times\RR^{2n-4}$, where $H_3\times\RR$ is the universal cover of the Kodaria-Thurston manifold.  It is easy to prove that $\lca(\omega_0)\cap\Gl(\ggo)\cdot\mu_0=\Spe(\omega_0)\cdot\mu_0$ (i.e. $(\ggo,\mu_0)$ admits a unique symplectic structure up to isomorphism).  Moreover, it is proved in the first example in \cite[Section 3]{praga} that $\Spe(\omega_0)\cdot\mu_0=\U(\omega_0,g_0)\cdot\mu_0$, from which follows that the Lie group $(H_3\times\RR)\times\RR^{2n-4}$ admits a unique left-invariant almost-K\"ahler structure up to equivalence for any $n\geq 2$.  Consequently, the solution starting at this structure will be self-similar for any curvature flow on almost-K\"ahler manifolds invariant by diffeomorphisms.
\end{example}

\subsection{Bracket flow}\label{sec-BF}
In view of the parametrization of left-invariant almost-K\"ahler structures as points in the variety $\lca(\omega_0)\subset\lca$ described in the above section, it is natural to study the dynamical system determined by SCF on $\lca(\omega_0)$.

Consider for a family $\mu(t)\in \Lambda^2\ggo^*\otimes\ggo$ of brackets the following evolution equation, called the {\it bracket flow}:
\begin{equation}\label{intro2}
\ddt\mu=\delta_\mu(P_\mu+\Ricci_\mu^{ac}), \qquad\mu(0)=\lb,
\end{equation}
where $P_\mu,\Ricci_\mu^{ac}\in\End(\ggo)$ are respectively the Chern-Ricci and Ricci operators of the almost-hermitian manifold $(G_\mu,\omega_0,g_0)$ and $\delta_\mu:\End(\ggo)\longrightarrow\Lambda^2\ggo^*\otimes\ggo$ is defined by
\begin{equation}\label{delta}
\delta_\mu(A):=\mu(A\cdot,\cdot)+\mu(\cdot,A\cdot)-A\mu(\cdot,\cdot), \qquad\forall A\in\End(\ggo).
\end{equation}
The bracket flow leaves the variety $\lca(\omega_0)$ invariant (i.e. $(G_{\mu(t)},\omega_0,g_0)$ is almost-K\"ahler for all $t$) and has been proved in \cite{SCF} to be equivalent to the SCF.

\begin{theorem}\label{BF-thm}\cite[Theorem 5.1]{SCF}
For a given simply connected almost-K\"ahler Lie group $(G,\omega_0,g_0)$ with Lie algebra $\ggo$, consider the families of almost-K\"ahler Lie groups
$$
(G,\omega(t),g(t)), \qquad (G_{\mu(t)},\omega_0,g_0),
$$
where $(\omega(t),g(t))$ is the solution to the SCF-flow starting at $(\omega_0,g_0)$ and $\mu(t)$ is the bracket flow solution starting at the Lie bracket $\lb$ of $\ggo$.  Then there exist Lie group isomorphisms $h(t):G\longrightarrow G_{\mu(t)}$ (i.e. $\mu(t)=h(t)\cdot\lb$) such that
$$
(\omega(t),g(t))=h(t)^*(\omega_0,g_0), \qquad\forall t.
$$
Moreover, the isomorphisms $h(t)$ can be chosen as the solution to the following systems of ODE's:
\begin{itemize}
\item[(i)] $\ddt h=-h(P+\Ricci^{ac})=-h(P^{ac}+\Ricci)$, $\quad h(0)=I$.
\item[ ]
\item[(ii)] $\ddt h=-(P_\mu+\Ricci_\mu^{ac})h=-(P_\mu^{ac}+\Ricci_\mu)h$, $\quad h(0)=I$.
\end{itemize}
\end{theorem}

The maximal interval of time existence $(T_-,T_+)$ is therefore the same for both flows, as it is the behavior
of any kind of curvature and so regularity issues can be addressed on the bracket flow.

The above theorem has also the following application on convergence, which follows from \cite[Corollary 6.20]{spacehm} (see \cite[Section 5.1]{SCF} for further information on convergence).

\begin{corollary}\label{conv}
Let $\mu(t)$ be a bracket flow solution and assume that $c_k\mu(t_k)\to\lambda$, for some nonzero $c_k\in\RR$ and a subsequence of times $t_k\to T_\pm$.  Then, after possibly passing to a subsequence, the almost-K\"ahler manifolds $\left(G,\tfrac{1}{c_k^2}\omega(t_k),\tfrac{1}{c_k^2}g(t_k)\right)$ converge in the pointed (or Cheeger-Gromov) sense to $(G_\lambda,\omega_0,g_0)$, as $k\to\infty$.
\end{corollary}

We note that the limiting Lie group $G_\lambda$ in the above corollary might be non-isomorphic, and even non-homeomorphic, to $G$ (see Examples \ref{6latt} and \ref{BF-exa}).

\subsection{Self-similar solutions}\label{sec-self}

In the general case, an almost-K\"ahler manifold $(M,\omega,g)$ will flow self-similarly along the SCF, in the sense that
$$
(\omega(t),g(t))=(c(t)\vp(t)^*\omega,c(t)\vp(t)^*g),
$$
for some $c(t)>0$ and $\vp(t)\in\Diff(M)$, if and only if
$$
\left\{
\begin{array}{l}
p(\omega,g)=c\omega+\lca_{X}\omega, \\ \\
p^c(\omega,g)(\cdot,J\cdot)+\ricci^{ac}(\omega,g)=cg+\lca_{X}g,
\end{array}\right.
$$
for some $c\in\RR$ and a complete vector field $X$ on $M$.  In analogy to the terminology used in Ricci flow theory, we call such $(\omega,g)$ a
{\it soliton almost-K\"ahler structure} and we say it is {\it expanding}, {\it steady} or {\it shrinking}, if $c<0$, $c=0$ or $c>0$, respectively.  On Lie groups, it is natural to consider a SCF-flow solution to be self-similar if the diffeomorphisms $\vp(t)$ above are actually Lie group automorphisms (this is actually a stronger condition, see \cite[Example 9.1]{SCF}).  It is proved in \cite[Section 7]{SCF} that this is equivalent to the following condition: we say that an almost-K\"ahler structure $(\omega,g)$ on a Lie algebra $\ggo$ is a {\it SCF-soliton} if for some $c\in\RR$ and $D\in\Der(\ggo)$,
\begin{equation}\label{SCF-sol}
\left\{
\begin{array}{l}
P=cI+\unm(D-JD^tJ), \\ \\
P^c+\Ricac=cI+\unm(D+D^t).
\end{array}\right.
\end{equation}
The following condition, suggested by the relationship between the SCF and the bracket flow given in Theorem \ref{BF-thm},
\begin{equation}\label{SCF-algsol}
P+\Ricac=cI+D,
\end{equation}
is sufficient to get a SCF-soliton (see \cite[Proposition 7.4]{SCF}) and an almost-K\"ahler structure for which this holds will be called an {\it algebraic SCF-soliton}, in analogy to the case of homogeneous Ricci solitons (see \cite[Section 3]{homRS} or \cite{Jbl}).  The bracket flow solution starting at an algebraic SCF-soliton is simply given by $\mu(t)=(-2ct+1)^{-1/2}\lb$ and hence they are precisely the fixed points and only possible limits, backward and forward, of any normalized bracket flow solution $c(t)\mu(t)$.  In particular, if in Corollary \ref{conv} one actually has that $c_t\mu(t)\to\lambda$, as $t\to T_\pm$, then the pointed limit $(G_\lambda,\omega_0,g_0)$ is an algebraic soliton.  The absence of certain chaotic behavior for the bracket flow would imply that any SCF-soliton is actually algebriac (see \cite[Section 7.1]{SCF}).

If an almost-K\"ahler structure $(\omega,g)$ satisfies that
\begin{equation}\label{strongly}
\left\{\begin{array}{l}
P=c_1I+D_1, \\ \\
\Ricci^{ac}=c_2I+D_2,
\end{array}\right.
\end{equation}
for some $c_i\in\RR$, $D_i\in\Der(\ggo)$, then $(\omega,g)$ is an algebraic SCF-soliton with $c=c_1+c_2$ and $D=D_1+D_2$.  We call these structures {\it strongly algebraic SCF-solitons}.  So far, all known examples of SCF-solitons are of this kind.

\begin{lemma}\label{cuni}
Let $(G,\omega,g)$ be a unimodular almost-hermitian Lie group such that $\Ricac = cI+D$ for some $c\in\RR$ and $D \in \Der(\ggo)$.  Then,
$$
cR = \tr{(\Ricci^{ac})^2},
$$
where $R=\tr{\Ricci}$ is the scalar curvature of $(G,g)$.
\end{lemma}

\begin{proof}
Since $\Ricac= cI +D$ anti-commute with $J$, we obtain that $D^{ac}=cI + D$ and $D^c = -cI$.  This implies that
\begin{equation}\label{cuni-eq}
\tr{\Ricci D} = \tr{\Ricci D^{c}} + \tr{\Ricci D^{ac}} =  -c \tr{\Ricci} + \tr{(\Ricac)^2},
\end{equation}
and so the lemma follows from the fact that $\tr{\Ricci D}=0$ when $\ggo$ is unimodular (see e.g. \cite[Remarks 2.4, 2.7]{alek}).
\end{proof}

It is well known that if $\ggo$ is unimodular and $\omega$ is closed, then $\ggo$ must be solvable (see \cite{LchMdn}), and any left-invariant metric $g$ on a solvable Lie group has $R\leq 0$, with equality $R=0$ holding if and only if $g$ is flat.

\begin{corollary}
Any unimodular strongly algebraic SCF-soliton $(G,\omega,g)$ as in {\rm \eqref{strongly}} with $g$ nonflat has $c_2\leq 0$, and $c_2=0$ if and only if $\Ricac=0$.
\end{corollary}

Lemma \ref{cuni} is no longer true if $\ggo$ is not unimodular, counterexamples can be easily found among the classes of structures studied in the next sections (see e.g. the soliton on $\rg_2´$ in Table \ref{n3}).  Anyway, formula \eqref{cuni-eq} can always be used in the non-unimodular case.

\section{Almost abelian solvmanifolds}\label{muA-sec}

We study in this section the SCF and its solitons in a class of solvable Lie algebras which is relatively simple from the algebraic point of view but yet geometrically rich and exotic.

Let $(G,\omega,g)$ be an almost-K\"ahler Lie group with Lie algebra $\ggo$ and assume that $\ggo$ has a codimension-one abelian ideal $\ngo$.  These Lie algebras are sometimes called {\it almost-abelian} in the literature (see e.g. \cite{Bck,CnsMcr}).    It is easy to see that there exists an orthonormal basis $\{ e_1,\dots,e_{2n}\}$ such that
$$
\ngo=\la e_1,\dots,e_{2n-1}\ra, \qquad \omega=e^1\wedge e^{2n}+\omega_1, \qquad
J=\left[\begin{array}{c|c|c} 0&0&-1 \\\hline &&\\ 0&\quad J_1\quad &0 \\ &&\\ \hline 1&0&0 \end{array}\right],
$$
where $\{ e^i\}$ denotes the dual basis, $\omega_1$ is a nondegenerate $2$-form on $\ngo_1:=\la e_2,\dots,e_{2n-1}\ra$ and $\omega_1=g(J_1\cdot,\cdot)$.  We fix in what follows the orthonormal basis $\{ e_i\}$ and the $2$-form $\omega$, thus obtaining a fixed euclidean symplectic vector space $(\ggo,\omega,g)$ which can be identified with $\RR^{2n}$.

Recall from Section \ref{sec-LG} the notation $\Spe(\omega)$, $\spg(\omega)$ and $\U(\omega,g)$.   We also use this notation for the $2$-form $\omega_1$ above and obtain $\Spe(\omega_1)$, $\spg(\omega_1)$ and $\U(\omega_1,g_1)$, where $g_1=g|_{\ngo_1}$, which are respectively isomorphic to $\Spe(n-1,\RR)$, $\spg(n-1,\RR)$ and $\U(n-1)$.

Each of these Lie algebras is therefore determined by the $(2n-1)\times(2n-1)$-matrix
$$
A:=\ad{e_{2n}}|_{\ngo},
$$
and so it will be denoted by $\mu_A$.  Thus $\mu_A$ is always solvable, $\ngo$ is always an abelian ideal (which is the nilradical of $\mu_A$ if and only if $A$ is not nilpotent) and $\mu_A$ is nilpotent if and only if $A$ is a nilpotent matrix.  It is easy to check that $\mu_A$ is isomorphic to $\mu_B$ if and only if $A$ and $B$ are conjugate up to a nonzero scaling.

\begin{proposition}\label{formA}
Any almost-K\"ahler Lie algebra with a codimension-one abelian ideal is equivalent to
$$
(\ggo,\mu_{A},\omega,g), \qquad
A=\left[\begin{array}{c|c}
a&v^t\\\hline &\\
0& \quad A_1 \quad \\ &\\
\end{array}\right],
$$
for some $a\geq 0$, $v\in\RR^{2n-2}$ and $A_1\in\spg(\omega_1)\simeq \spg(n-1,\RR)$ (i.e. $A_1^tJ_1+J_1A_1=0$).
\end{proposition}

\begin{proof}
By using that the only nonzero Lie brackets are the ones involving $e_{2n}$, it is easy to see that $\omega$ is closed (see \eqref{closed}) if and only if
$$
\omega(Ae_i,e_j) - \omega(Ae_j,e_i) = 0, \qquad\forall i,j\ne 2n,
$$
which is equivalent to $Ae_1\in\RR e_1$ and $\omega_1(A_1\cdot,\cdot)+\omega_1(\cdot,A_1\cdot)=0$.  Thus $A_1\in\spg(\omega_1)$ and $Ae_1=ae_1$ for some $a\in\RR$, which can be assumed nonnegative by changing to the basis $\{ e_1,\dots,-e_{2n}\}$ if necessary.
\end{proof}

\begin{remark}\label{J1}
It can be assumed that $J_1=\left[\begin{array}{c|c}
0&-I\\\hline
I& 0
\end{array}\right]$, in which case
$$
\spg(\omega_1)=\left\{\left[\begin{array}{c|c}
B&C\\\hline
D&-B^t
\end{array}\right] : C^t=C, \quad D^t=D\right\}.
$$
\end{remark}

The almost-K\"ahler Lie algebra $(\ggo,\mu_A,\omega,g)$ in Proposition \ref{formA} determines a left-invariant almost-K\"ahler structure on the corresponding simply connected Lie group $G_{\mu_A}$, which will be denoted by $(G_{\mu_A},\omega,g)$.

Let $\mu_B$ be another Lie algebra as above, where
$$
B=\left[\begin{array}{c|c}
b&w^t\\\hline &\\
0& \quad B_1 \quad \\ &\\
\end{array}\right], \qquad b\geq 0, \quad w\in\RR^{2n-2}, \quad B_1\in\spg(\omega_1).
$$

\begin{proposition}\label{iso-equiv}
Let $A$, $B$ be two matrices as above and assume that neither is nilpotent.

\begin{itemize}
\item[(i)] The symplectic Lie algebras $(\ggo,\mu_A,\omega)$ and $(\ggo,\mu_B,\omega)$ are isomorphic if and only if there exists $\alpha\ne 0$, $\vp_1\in\Spe(\omega_1)$ and $u\in\ngo_1$ such that
    $$
    b=\alpha a, \qquad B_1=\alpha\vp_1A_1\vp_1^{-1}, \qquad w=\alpha^2(\vp_1^t)^{-1}\left(v+(A_1^t-aI)J_1\vp_1^{-1}u\right).
    $$
\item[(ii)] The almost-K\"ahler structures $(G_{\mu_A},\omega,g)$ and $(G_{\mu_B},\omega,g)$ are equivalent if and only if $b=a$ and there exists $\vp_1\in\U(\omega_1,g_1)\simeq\U(n-1)$ such that
$$
B_1=\vp_1A_1\vp_1^{-1} \quad (B_1=\pm\vp_1A_1\vp_1^{-1} \;\mbox{if}\; a=b=0), \qquad w=\vp_1v.
$$
\end{itemize}
\end{proposition}

\begin{proof}
To prove part (i), we first recall from Section \ref{sec-LG} that these symplectic Lie algebras are isomorphic if and only if $\mu_B=\vp\cdot\mu_A$ for some $\vp\in\Spe(\omega)$.  It is easy to see by using that $\vp$ leaves $\ngo$ invariant (notice that $\ngo$ is the nilradical of both Lie algebras as $A$ and $B$ are not nilpotent), that such a $\vp$ must have the form
$$
\vp=\left[\begin{array}{c|c|c}
\alpha & \alpha(J_1\vp_1^{-1}u)^t & \beta\\\hline & & \\
0 & \vp_1 & u \\ & & \\ \hline
0 & 0 & \alpha^{-1}
\end{array}\right], \qquad \mbox{for some} \quad \alpha,\beta\in\RR, \quad u\in\RR^{2n-2}, \quad \vp_1\in\Spe(\omega_1).
$$
Condition $\mu_B=\vp\cdot\mu_A$ is now equivalent to $B=\alpha\vp|_{\ngo} A(\vp|_{\ngo})^{-1}$, from which part (i) easily follows.

We now prove part (ii).  Since the structures are equivalent if and only if there exists $\vp\in\U(\omega,g)$ such that $\mu_B=\vp\cdot\mu_A$, we obtain from part (i) that $\vp$ has the form above with $\beta=0$, $u=0$, $\alpha=\pm 1$ and $\vp_1\in\U(\omega_1,g_1)$.  Thus $a=b$ since $a,b\geq 0$ and $\alpha=-1$ is only allowed when $a=b=0$, concluding the proof.
\end{proof}

\begin{remark}\label{scaling}
It follows that if
$$
A_\alpha:=\left[\begin{array}{c|c}
\alpha a&\alpha^2 v\\ \hline &\\
0& \quad \alpha A_1 \quad \\ &\\
\end{array}\right], \qquad \alpha>0,
$$
then $(G_{\mu_{A_\alpha}},\omega,g)$ is equivalent to the almost-K\"ahler Lie group $(G_{\mu_A},\omega, g_\alpha)$, where $g_\alpha$ is the inner product defined by $g_\alpha(e_1,e_1)=\alpha^2$, $g_\alpha(e_{2n},e_{2n})=\alpha^{-2}$, $g_\alpha(e_1,e_{2n})=0$ and $g_\alpha(e_i,e_j)=\delta_{ij}$ for all $2\leq i,j\leq 2n-1$.  In the case $v=0$, $\mu_{A_\alpha}=\alpha\mu_A$ and $(G_{\mu_{A_\alpha}},\omega,g)$ is also equivalent to the almost-K\"ahler Lie group $(G_{\mu_A},\alpha^{-2}\omega, \alpha^{-2}g_\alpha)$.
\end{remark}

\begin{example}\label{nilpmat-exa}
If $n=3$ and $A$, $B$ are defined by taking $a=b=0$, $v=w=0$ and
\begin{equation}\label{nilpmat}
A_1=\left[\begin{array}{c|c}
\begin{smallmatrix} 0&0\\ 0&0\end{smallmatrix} & \begin{smallmatrix} 0&0\\ 0&1\end{smallmatrix} \\\hline
\begin{smallmatrix} 0&0\\ 0&0\end{smallmatrix} & \begin{smallmatrix} 0&0\\ 0&0\end{smallmatrix}
\end{array}\right], \qquad
B_1=\left[\begin{array}{c|c}
\begin{smallmatrix} 0&0\\ 0&0\end{smallmatrix} & \begin{smallmatrix} 0&0\\ 0&-1\end{smallmatrix} \\\hline
\begin{smallmatrix} 0&0\\ 0&0\end{smallmatrix} & \begin{smallmatrix} 0&0\\ 0&0\end{smallmatrix}
\end{array}\right],
\end{equation}
then the Lie algebras $\mu_A$ and $\mu_B$ are isomorphic but the symplectic Lie algebras $(\ggo,\mu_A,\omega)$ and $(\ggo,\mu_B,\omega)$ are not.  Indeed, $A_1$ and $B_1$ are $\Gl_4(\RR)$-conjugate but they belong to different $\Spe(2,\RR)$-conjugacy classes.
\end{example}

It follows from \cite[(8)]{Arr} that the Ricci operator of $(G_{\mu_A},g)$ is given by
\begin{equation}\label{RicA}
\begin{array}{ccl}
\Ricci & = & \left[\begin{array}{c|c}
& \\
\unm[A, A^t]-a S(A) &  0 \\ & \\\hline
0 & -\tr S(A)^2
\end{array}\right] \\ \\
& = & \left[\begin{array}{c|c|c}
-a^2 +\unm |v|^2 & (\unm A_1v-av)^t & 0\\\hline & & \\
\unm A_1v-av & \unm [A_1, A_1^t]-\unm vv^t - aS(A_1) & 0 \\ & & \\ \hline
0 & 0 & -a^2-\unm |v|^2-\tr{S(A_1)^2}
\end{array}\right],
\end{array}
\end{equation}
and a straightforward computation shows that its anti-J-invariant part is
\begin{equation}\label{RicacA}
\Ricac = \left[\begin{array}{c|c|c}
\unm\left(|v|^2 + \tr{S(A_1)^2}\right) & \left(\unc A_1v-\tfrac{a}{2}v\right)^t  & 0 \\\hline &&\\
\unc A_1v-\tfrac{a}{2}v & \unm [A_1, A_1^t]- aS(A_1) - \unm(vv^t)^{ac} & J_1\left(\unc A_1v-\tfrac{a}{2}v\right) \\&& \\\hline
0 & \left(J_1\left(\unc A_1v-\tfrac{a}{2}v\right)\right)^t & -\unm\left(|v|^2 + \tr{S(A_1)^2}\right)
\end{array}\right].
\end{equation}
The scalar curvature of $(G_{\mu_A},g)$ is therefore given by
$$
R=-a^2-\tr{S(A)^2} = -2a^2-\unm|v|^2-\tr{S(A_1)^2}.
$$
By using \eqref{CRform}, it is straightforward to obtain that the Chern-Ricci operator of $(G_{\mu_A},\omega,g)$ is given by
\begin{equation}\label{PA}
P=\left[\begin{array}{c|c|c}
-a^2 & -\left(\unm A_1^tv+av\right)^t & 0\\\hline
 & & \\
0 & 0 & -J_1\left(\unm A_1^tv+av\right) \\
 & & \\ \hline
0 & 0 & -a^2
\end{array}\right],
\end{equation}
and thus the Chern scalar curvature is $\tr{P}=-2a^2$.

We note that the following conditions are equivalent:

\begin{itemize}
\item $(G_{\mu_A},\omega,g)$ is K\"ahler.

\item $\Ricac=0$.

\item $v=0$ and $A_1^t=-A_1$ (i.e. $A_1\in\sug(n-1)$).

\item $R=\tr{P}$.

\item $(G_{\mu_A},\omega,g)$ is either equivalent as an almost-K\"ahler manifold (not as a Lie group) to $\RR H^2\times\RR^{2n-2}$, where $\RR H^2$ denotes the $2$-dimensional real hyperbolic space ($a>0$) or to the euclidean space $\RR^{2n}$ ($a=0$).
\end{itemize}

The equivalence between the first and third conditions above also follows from Proposition \ref{formA} and \cite[Lemma 6.1]{CRF}.

\subsection{SCF-solitons}
We now explore necessary and sufficient conditions on the matrix $A$ to obtain a SCF-soliton $(G_{\mu_A},\omega,g)$.

\begin{theorem}\label{SA-sol}
Let $(G_{\mu_A},\omega,g)$ denote the almost-K\"ahler structure defined as in Proposition \ref{formA}.
\begin{itemize}
\item[(i)] If $A$ is not nilpotent, then $\mu_A$ is an algebraic SCF-soliton if and only if $v=0$ and $A_1$ is normal, if and only if $A$ is normal.
\item[(ii)] If $v=0$, then $\mu_A$ is an algebraic SCF-soliton if and only if either $A$ is normal or $A$ is nilpotent (i.e. $a=0$ and $A_1$ nilpotent) and
\begin{equation}\label{nilsol}
[A_1,[A_1,A_1^t]]=-\frac{|[A_1,A_1^t]|^2}{|A_1|^2}A_1.
\end{equation}
\end{itemize}
\end{theorem}

\begin{remark}\label{allexp}
It is easy to check that all the (non-flat) SCF-solitons obtained in this theorem are strongly algebraic and expanding.  Indeed, if $A$ is normal then $c_1=-a^2$ $c_2=-\unm\tr{S(A_1)^2}$ and so $c=-(a^2+\unm\tr{S(A_1)^2})$, and in the case when $A$ is nilpotent, $P=0$ and $\Ricac=c_2I+D_2$ for $c_2=c=-\unm\left(\tfrac{|[A_1,A_1^t]|^2}{|A_1|^2}+\tr{S(A_1)^2}\right)$.
\end{remark}

\begin{proof}
We first prove part (i).  Since a linear map $D:\ggo\longrightarrow\ggo$ is a derivation of $\mu_A$ if and only if its image is contained in $\ngo$ and $[D|_{\ngo}, A]=0$ (recall that $\ngo$ is the nilradical of $\ggo$ when $A$ is not nilpotent), we obtain from \eqref{RicacA} and \eqref{PA} that $P+\Ricac-cI$ is a derivation of $\mu_A$ for some $c\in\RR$ if and only if $c=-a^2-\unm\left(|v|^2 + \tr S(A_1)^2\right)$ and
\begin{align}
A_1v = &2av, \label{alg1} \\
(A_1^t)^2v + A_1A_1^tv - 2aA_1^tv = &\left(\tfrac{3}{2}|v|^2 + \tr S(A_1)^2+2a^2\right)v, \label{alg2} \\
[A_1,[A_1, A_1^t]] -a[A_1,A_1^t] = & [A_1, (vv^t)^{ac}]. \label{alg3}
\end{align}
By multiplying scalarly equation \eqref{alg2} by $v$ and $J_1v$ and using \eqref{alg1} we respectively obtain,
\begin{align}
|A_1^tv|^2 =& \left(\tfrac{3}{2}|v|^2 + \tr{S(A_1)^2}+2a^2\right)|v|^2, \label{alg4} \\
-4a\la A_1^tv,J_1v\ra =& 0, \label{alg6}
\end{align}
If $a\ne 0$, then $\la A_1^tv,J_1v\ra=0$ by \eqref{alg6} and thus $A_1^tv=2av+w$ with $w$ orthogonal to $\{ v,J_1v\}$.  This implies that $A_1J_1v=-2aJ_1v-J_1w$ and thus
$$
\tr{S(A_1)^2}|v|^2 \geq 8a^2 |v|^2+ |w|^2 > 4a^2|v|^2 + |w|^2 = |A_1^tv|^2,
$$
which contradicts equation \eqref{alg4} unless $v=0$.  It follows that $A_1$ is normal by multiplying scalarly equation \eqref{alg3} by $A_1$.

We therefore assume that $a=0$.  By using that $(vv^t)^{ac}v=\unm|v|^2v$, $(vv^t)^{ac}Jv=-\unm|v|^2Jv$ and $(vv^t)^{ac}$ vanishes on the orthogonal complement of $\{ v,Jv\}$, one obtains
\begin{equation}\label{alg5}
\tr{[A_1,A_1^t](vv^t)^{ac}} = |A_1^tv|^2.
\end{equation}
It now follows from \eqref{alg3}, \eqref{alg5} and \eqref{alg4} that
\begin{align*}
|[A_1,A_1^t]-vv^t|^2 =& -\la A_1,[A_1,[A_1,A_1^t]]\ra + |v|^4 - 2\la A_1A_1^t,vv^t\ra + 2\la A_1^tA_1,vv^t\ra \\
=& -\tr{A_1^t[A_1,(vv^t)^{ac}]} + |v|^4 - 2|A_1^tv|^2 \\
=& \tr{[A_1,A_1^t](vv^t)^{ac}} + |v|^4 - 2|A_1^tv|^2 \\
=& -|A_1^tv|^2 + |v|^4 \\
=& -\left(\tfrac{3}{2}|v|^2 + \tr{S(A_1)^2} \right)|v|^2 + |v|^4  \\
=& (-\unm|v|^2 - \tr{S(A_1)^2})|v|^2,
\end{align*}
and therefore $v=0$ and $A_1$ is normal.

To prove part (ii), we can assume that $A$ is nilpotent by part (i).  Since $v=0$ $P+\Ricac$ has a block diagonal form and so it is easy to check that $D:=P+\Ricac-cI$ is a derivation of $\mu_A$ for some $c\in\RR$ if and only if $[D,\ad{e_{2n}}]=\la De_{2n},e_{2n}\ra\ad{e_{2n}}$, which is equivalent to $[A_1,[A_1,A_1^t]]$ being a scalar multiple of $A_1$.  The multiple can be computed by multiplying scalarly by $A_1^t$, concluding the proof.
\end{proof}

\begin{example}
By defining
$$
A_r:=\left[\begin{array}{c|c}
1&0\\\hline
0& \begin{array}{c|c}
rI&0\\\hline
0&-rI
\end{array} \\
\end{array}\right],
$$
we obtain, in any dimension $\geq 4$, a one-parameter family of pairwise non-equivalent expanding SCF-solitons $(G_{\mu_{A_r}},\omega,g)$ (see Theorem \ref{SA-sol}, (i)) with
$$
P=\left[\begin{array}{c|c|c}
-1 & 0 & 0\\\hline
0 & 0 & 0 \\
\hline
0 & 0 & -1
\end{array}\right], \qquad
\Ricac=\left[\begin{array}{c|c|c}
-(n-1)r^2 & 0 & 0\\\hline
0 & \begin{array}{c|c}
-rI&0\\\hline
0&rI
\end{array} & 0 \\
\hline
0 & 0 & -(n-1)r^2
\end{array}\right].
$$
We note that actually the Lie algebras $\mu_{A_r}$ are pairwise non-isomorphic.
\end{example}

\begin{example}\label{ex-Anilp}
Consider the almost-K\"ahler structure $(G_{\mu_A},\omega,g)$ with $a=0$, $v=0$ and
$$
A_1=\left[\begin{array}{cc} 0&0\\ C&0\end{array}\right], \qquad C^t=C.
$$
It is straightforward to check that the soliton condition \eqref{nilsol} in Theorem \ref{SA-sol}, (ii) holds for $A_1$ if and only if $C^3=\tfrac{\tr{C^4}}{\tr{C^2}}C$.  We can assume, up to isometry, that $C$ is diagonal (see the proof of Proposition \ref{iso-equiv}).  In that case, $(G_{\mu_A},\omega,g)$ is an algebraic SCF-soliton if and only if any diagonal entry of $C$ is either equal to $0$, $1$ or $-1$ (compare with Example \ref{nilpmat-exa}).
\end{example}

In what follows, we study under what conditions on $A$ the symplectic Lie group $(G_{\mu_A},\omega)$ admits a compatible left-invariant metric such that the corresponding almost-K\"ahler structure is a SCF-soliton.  According to the observation made at the end of Section \ref{sec-LG} that $\Spe(\omega)\cdot\mu_A$ parameterizes the set of all compatible metrics on $(G_{\mu_A},\omega)$ and Proposition \ref{iso-equiv}, (i), this is equivalent to the existence of a matrix $B$ satisfying the conditions in the proposition and such that $\mu_B$ is a SCF-soliton.  We note that the uniqueness up to equivalence of the SCF-soliton metric can be analyzed by using  Proposition \ref{iso-equiv}, (ii).

The following corollary of Theorem \ref{SA-sol}, (i) therefore follows from the fact that a matrix is semisimple (always understood over the complex numbers) if and only if it is conjugate to a normal matrix.

\begin{corollary}\label{SA-cor}
If $A$ is neither nilpotent nor semisimple, then the Lie group $G_{\mu_A}$ does not admit any algebraic SCF-soliton.
\end{corollary}

We now give some existence results for SCF-solitons.

\begin{proposition}\label{existss}
If $v=0$ and $A$ is semisimple, then the symplectic Lie group $(G_{\mu_{A}},\omega)$ admits a compatible metric $g$ such that the almost-K\"ahler structure $(\omega,g)$ is an algebraic SCF-soliton.  Moreover, any other algebraic SCF-soliton $(\widetilde{\omega},\widetilde{g})$ on $G_{\mu_A}$ such that the symplectic structure $\widetilde{\omega}$ is isomorphic to $\omega$ is equivalent to $(\omega,g)$ up to scaling.
\end{proposition}

\begin{remark}
The uniqueness statement in the proposition does not imply that there is a unique algebraic SCF-soliton $g$ on $(G_{\mu_A},\omega)$ up to equivalence (see Remark \ref{scaling}).
\end{remark}

\begin{proof}
If $A$ is semisimple then $A_1$ is a semisimple element in $\spg(\omega_1)$ and it is well-known that so there exists $\vp_1\in\Spe(\omega_1)$ such that $\vp_1A_1\vp_1^{-1}$ is normal.  This implies that $\vp\cdot\mu_A$ is a an algebraic SCF-soliton, where $\vp\in\Spe(\omega)$ is defined by $\vp|_{\{ e_1,e_{2n}\}}=id$, $\vp|_{\ngo_1}=\vp_1$ (see Theorem \ref{SA-sol}, (ii)).

The uniqueness up to equivalence and scaling follows from the fact that the subset of normal matrices in the $\Spe(\omega_1)$-conjugacy class of $A_1$ consists of a single $\U(\omega_1,g_1)$-orbit.  Indeed, if $\mu_B=\psi\cdot\mu_A$ with $\psi\in\Spe(\omega)$ is another algebraic SCF-soliton, then from Proposition \ref{iso-equiv}, (i) we obtain that $b=\alpha a$ and $B_1=\alpha \psi_1A_1\psi_1^{-1}$.  In particular $B$ is not nilpotent and hence $w=0$ and $B_1$ is normal by Theorem \ref{SA-sol}, (i).  This implies that there exists $h_1\in\U(\omega_1,g_1)$ such that $B_1=\alpha h_1\vp_1A_1\vp_1^{-1}h_1^{-1}$ and thus $\mu_{B}$ is equivalent to $\alpha\vp\cdot\mu_A$ (see Proposition \ref{iso-equiv}, (ii)), which is by Remark \ref{scaling} equivalent to the almost-K\"ahler structure $(\alpha^{-2}\vp^*\omega,\alpha^{-2}\vp^*g)$ on $G_{\mu_A}$, concluding the proof.
\end{proof}

The case $v=0$ and $A$ nilpotent is more involved, we shall need some results from \cite{Jbl2} on geometric invariant theory concerning moment maps for real representations of real reductive Lie groups (see e.g. \cite[Appendix]{cruzchica} for more information).

\begin{proposition}\label{existnil}
If $v=0$ and $A$ is nilpotent, then the symplectic Lie group $(G_{\mu_{A}},\omega)$ admits a compatible metric $g$ such that the almost-K\"ahler structure $(\omega,g)$ is an algebraic SCF-soliton.
\end{proposition}

\begin{proof}
It is known that condition \eqref{nilsol} holds for a nilpotent matrix $A_1$ if and only if $A_1$ is a critical point of the functional square norm of the moment map $F(B):=|m(B)|^2$.  Here $m:\glg_{2n-2}(\RR)\longrightarrow\sym(2n-2)$ is the moment map for the $\Gl_{2n-2}(\RR)$-action by conjugation on $\glg_{2n-2}(\RR)$ and is given by $m(B)=\frac{[B,B^t]}{|B|^2}$.  It follows from \cite[Theorem 4.2]{finding} (see also \cite[Proposition 5.4]{Jbl2}) that each nilpotent conjugacy class contains a critical point whose $k\times k$-Jordan blocks are given by
$$
\left[\begin{array}{cccc}
0&b_1&      &       \\
 &0  &\ddots&       \\
 &   &\ddots&b_{k-1} \\
 &   &      &0
\end{array}\right], \qquad b_i=\sqrt{i(k-i)}.
$$
These special matrices are the minima of $F$ on the conjugacy class.  From the general theory of moment maps we know that critical points of $F$ are unique up to the action of the maximal compact subgroup $\Or(2n-2)$ of $\Gl_{2n-2}(\RR)$.

Following the notation of \cite{Jbl2}, we take
$$
G=\Gl_{2n-2}(\RR), \quad V=\glg_{2n-2}(\RR), \quad H=\Spe(n-1,\RR), \quad W=\spg(n-1,\RR).
$$
By \cite[Theorem 3.1]{Jbl2} we have that the intersection of the $\Gl_{2n-2}(\RR)$-conjugacy class of each $A_1\in\spg(n-1,\RR)$ with $\spg(n-1,\RR)$ is a finite union of $\Spe(n-1)$-orbits (more than one in general, see e.g. \eqref{nilpmat}).  Moreover, it follows from \cite[Corollary 3.4]{Jbl2} that each of these $\Spe(n-1,\RR)$-conjugacy classes contains a unique up to $\U(n-1)$-conjugation critical point $\widetilde{A_1}$ of the $\Spe(n-1,\RR)$-moment map, which coincides with the $\Gl_{2n-2}(\RR)$-moment map on $\spg(n-1)$ and so $\widetilde{A_1}$ satisfies condition \eqref{nilsol}.  Since $\mu_{\widetilde{A_1}}=\vp\cdot\mu_A$ for a suitable $\vp\in\Spe(\omega)$ we conclude that $(G_{\mu_A},\omega)$ admits an algebraic SCF-soliton by Theorem \ref{SA-sol}, (ii), as was to be shown.
\end{proof}

\subsection{Dimension $4$}\label{muA-dim4-sec}
We now consider the almost abelian case when $\dim{\ggo}=4$, i.e. $n=2$.  If $\{ e_1, \dots, e_4\}$ is the canonical basis of $\ggo\equiv \RR^4$, then we fix
$$
\omega = e_1 \wedge e_4 + e_2 \wedge e_3, \qquad J= \left[\begin{smallmatrix} &  &  & -1 \\  & & -1 &  \\  & 1 &  & \\ 1& & & \end{smallmatrix}\right].
$$
Since the $2\times 2$ matrix $A_1$ is symplectic if and only if $\tr{A_1}=0$ (see Proposition \ref{formA}), the matrices $A$ have the form
\begin{equation}\label{Adim4}
A= \left[\begin{matrix}a & b & c \\  0 & d & e \\ 0 & f &-d \end{matrix}\right], \qquad a,b,c,d,e,f  \in \RR.
\end{equation}

It follows from Theorem \ref{SA-sol}, (i) that if $A$ is not nilpotent, then $\mu_A$ is an algebraic SCF-soliton if and only if $A$ has one of the following two forms:
$$
\left[\begin{matrix}a & 0 & 0 \\  0 & d & e \\ 0 & e &-d \end{matrix}\right], \qquad \left[\begin{matrix}a & 0 & 0 \\  0 & 0 & e \\ 0 & -e &0 \end{matrix}\right].
$$

\begin{lemma}\label{r40-r4-1}
The Lie algebras $\rg_{4,0}$ and $\rg_{4,-1}$ (see Table \ref{n1}) do not admit any algebraic SCF-soliton.
\end{lemma}

\begin{proof}
These Lie algebras are isomorphic to $\mu_A$, where $A$ is respectively given by
$$
\left[\begin{matrix}1 & 0 & 0 \\  0 & 0 & 1 \\ 0 & 0 &0 \end{matrix}\right], \qquad \left[\begin{matrix}1 & 0 & 0 \\  0 & -1 & 1 \\ 0 & 0 &-1 \end{matrix}\right].
$$
The result therefore follows from Corollary \ref{SA-cor}, as these matrices are neither semisimple nor nilpotent.
\end{proof}

Any other $4$-dimensional symplectic Lie algebra isomorphic to a $\mu_A$ does admit an algebraic SCF-soliton which has been explicitly given in Table \ref{n3}.  This follows from a direct application of Theorem \ref{SA-sol} and Propositions \ref{existss}, \ref{existnil}, with the only exception of $\ngo_4$.

\subsection{Bracket flow}\label{BFmuA-sec}
We study in this section bracket flow evolution of almost-K\"ahler structures $(G_{\mu_A},\omega,g)$ (see Section \ref{sec-BF}).  We first introduce the following notation for each matrix $A$ as in Proposition \ref{formA}:
$$
Q_A:=P_{\mu_A}+\Ricci_{\mu_A}^{ac}, \qquad \lambda_A:=\delta_{\mu_A}\left( Q_A \right), \qquad r:=\unc A_1v-\tfrac{a}{2}v, \qquad c:=\unm A_1^tv+av,
$$
$$
\alpha_1:=-a^2+\unm(|v|^2 + \tr{S(A_1)^2}), \qquad \alpha_2:=-a^2-\unm(|v|^2 + \tr{S(A_1)^2}).
$$
It follows from \eqref{RicacA} and \eqref{PA} that
\begin{equation}\label{QA}
Q_A=\left[\begin{array}{c|c|c}
\alpha_1 & r^t-c^t & 0\\\hline
 & & \\
r & Q_1 & J_1(r-c) \\
 & & \\ \hline
0 & (J_1r)^t & \alpha_2
\end{array}\right],
\end{equation}
where $Q_1:=\unm [A_1, A_1^t]- aS(A_1) - \unm(vv^t)^{ac}$, and a straightforward computation gives
\begin{align}
\lambda_A(e_1,e_i) =& a\la r,J_1e_i\ra e_1, \qquad\forall i \ne 1, 2n, \label{bf1}
\\ \notag\\
\lambda_A(e_i,e_j) =& -\la r,J_1e_i\ra (v_je_1+A_1e_j) + \la r,J_1e_j\ra (v_ie_1+A_1e_i), \qquad \forall i,j \ne 1, 2n, \label{bf2}
\\ \notag \\
\lambda_A(e_{2n},e_1) =& \left(-\unm(2a^2+|v|^2 + \tr{S(A_1)^2})a +\la r,v\ra\right)e_1 + A_1r -ar, \label{bf3}
\\ \notag \\
\lambda_A(e_{2n},e_i) =& -\unm(2a^2+|v|^2 + \tr{S(A_1)^2})Ae_i + [A,Q]e_i, \notag
\\
=&   \left\langle -(|v|^2 + \tr{S(A_1)^2})v + Q_1v + a(r-c) - A_1^t(r-c),e_i\right\rangle e_1 \label{bf4}
\\
& -\unm(2a^2+|v|^2 + \tr{S(A_1)^2})A_1e_i + [A_1,Q_1]e_i-v_ir,   \qquad\forall i \ne 1, 2n. \notag
\end{align}

In order to get an invariant family under the bracket flow $\ddt\mu_A=\lambda_A$, we need to have $\lambda_A = \mu_B$ for some matrix $B$ of the same form as $A$ in Proposition \ref{formA} for all $t$, that is, $\lambda_A(\ngo_1,\ngo_1)=0$ (see \eqref{bf1} and \eqref{bf2}) and $\lambda_A(e_{2n},e_1)\in\RR e_1$ (see \eqref{bf3}).  Note that conditions $\lambda_A(e_{2n},\ngo)\subset\ngo$ and $\ad_{\lambda_A}{e_{2n}|_{\ngo_1}}\in\spg(\omega_1)$ automatically hold (see \eqref{bf4}).  When $v=0$ this clearly holds and the evolution will be studied below.

We therefore assume that $v\ne 0$.  If $r\ne 0$ then $a=0$ by \eqref{bf1}, and since the vectors $-v_je_i+v_ie_j$, $2\leq i,j\leq 2n-1$, generate the orthogonal complement $v^\perp$ of $v$ it follows from \eqref{bf2} that $r\in\RR J_1v$.  Moreover, \eqref{bf2} implies that $A_1v^\perp=0$ if $r\ne 0$, and so $\mu_A$ is isomorphic to $\hg_3\oplus\RR^{2n-3}$ as a Lie algebra and $(G_{\mu_A},\omega,g)$ is equivalent to its unique almost-K\"ahler structure (see Example \ref{kodaira}).

On the other hand, if $r=0$, then the four equations above give that $\lambda = \mu_B$ and the bracket flow equation for $A=A(t)$ would become
\begin{align}
  a' =& -\unm(2a^2 + |v|^2 + \tr{S(A_1)^2})a, \label{eva}
  \\ \notag \\
  v' =& -(2a^2 + \tfrac{5}{4}|v|^2 + \tr{S(A_1)^2})v + \unm A_1A_1^tv + \unm(A_1^t)^2v - aA_1^tv, \label{evv}
  \\ \notag \\
  A_1' =& -\unm(2a^2 + |v|^2 + \tr{S(A_1)^2})A_1 +[A_1,Q_1] \label{evA1}
  \\
  =& -\unm(2a^2 + |v|^2 + \tr{S(A_1)^2})A_1 + \unm[A_1,[A_1, A_1^t]] \notag
  \\
  &-\tfrac{a}{2}[A_1,A_1^t] - \unm[A_1,(vv^t)^{ac}], \label{evA2}
\end{align}
Unfortunately, condition $r=0$, which is equivalent to $A_1v=2av$, is not invariant under this ODE system and hence we need to consider smaller subsets to get invariant families under the bracket flow in the case $v\ne 0$.

\begin{proposition}\label{invf1}
The family $\left\{\mu_A : a=0,\quad A_1v=0, \quad A_1^tv=0\right\}$ is invariant under the bracket flow, which becomes equivalent to
\begin{align*}
    v' =& -(\tfrac{5}{4}|v|^2 + \tr{S(A_1)^2})v,
  \\
  A_1' =& -\unm(|v|^2 + \tr{S(A_1)^2})A_1 + \unm[A_1,[A_1, A_1^t]].
\end{align*}
\end{proposition}

\begin{remark}
The Chern-Ricci operator $P$ vanishes for any structure in this family as $r=c=0$.  Thus the SCF-evolution reduces to the anti-complexified Ricci flow (i.e. the symplectic structure remains fixed).
\end{remark}

\begin{proof}
The evolution equations for $v$ and $A_1$ follow from \eqref{evv} and \eqref{evA2}, respectively.  By using them, it is straightforward to compute the evolution of the vectors $A_1v$ and $A_1^t v$ to show that they remain zero in time, concluding the proof.
\end{proof}

\begin{proposition}\label{invf2}
The family $\{\mu_A : a=0,\quad A_1v=0, \quad A_1^2=0\}$ is invariant under the bracket flow, which becomes equivalent to
\begin{align*}
    v' =& -(\tfrac{5}{4}|v|^2 + \tr{S(A_1)^2})v + \unm A_1A_1^tv,
  \\
  A_1' =& -\unm(|v|^2 + \tr{S(A_1)^2})A_1 + \unm[A_1,[A_1, A_1^t]] - \unm[A_1,(vv^t)^{ac}].
\end{align*}
\end{proposition}

\begin{remark}
Each Lie algebra in this family is either $2$-step ($A_1^tv=0$) or $3$-step nilpotent ($A_1^tv\ne 0$).  The Chern-Ricci operator $P$ does not vanish in the $3$-step case; however, $P$ is always a derivation and so the SCF-evolution also reduces to the anti-complexified Ricci flow as for the above family.  It is easy to see that the SCF-solutions given in \cite[Example 9.4]{SCF} belong to this family.
\end{remark}

\begin{proof}
The evolution equations for $v$ and $A_1$ follow from \eqref{evv} and \eqref{evA2}, respectively.  It is then easy to compute the evolution of the vector $A_1v$ and the matrix $A_1^2$ to conclude that they remain zero in time, concluding the proof.
\end{proof}

\subsection{The case $v=0$}
The subset $\{\mu_A:v=0\}$ is invariant under the bracket flow, in the sense that any bracket flow solution starting at one of these structures has the form $\mu_{A(t)}$.  Since for each $t$ the Lie algebra $\mu_{A(t)}$ is isomorphic to the starting point $\mu_{A_0}$, we have that
$$
A(t)=c(t)H(t)A_0H(t)^{-1}, \qquad\mbox{for some} \quad c(t)>0, \quad H(t)\in\Gl_{2n-1}(\RR).
$$
The corresponding spectra (i.e. the unordered set of complex eigenvalues) therefore satisfy
\begin{equation}\label{spec}
\Spec(A(t))=c(t)\Spec(A_0), \qquad\forall t.
\end{equation}
It follows from \eqref{eva} and \eqref{evA2} that the bracket flow is equivalent when $v=0$ to the ODE system for $a=a(t)$ and $A_1=A_1(t)$ given by
\begin{equation}\label{BFv0}
\left\{\begin{array}{l}
a'=-(a^2+\unm\tr{S(A_1)^2})a, \\ \\
A_1'=  -(a^2+\unm\tr{S(A_1)^2})A_1+\unm[A_1,[A_1,A_1^t]]-\tfrac{a}{2}[A_1,A_1^t].
\end{array}\right.
\end{equation}

By using that $a=\tr{A}$ and $\tr{S(A)^2}=a^2+\tr{S(A_1)^2}$, this system can be written as a single equation for $A=A(t)$ as follows,
\begin{equation}\label{BFv02}
A'=  -\unm(a^2+\tr{S(A)^2})A+\unm[A,[A,A^t]]-\tfrac{\tr{A}}{2}[A,A^t].
\end{equation}

This equation differs from the bracket flow \cite[(7)]{Arr} used by Arroyo to study the Ricci flow for Riemannian manifolds $(G_{\mu_A},g)$ only in the coefficient that multiplies $A$, which is $-\tr{S(A)^2}$ in that case.  One therefore obtains, with identical proofs as in \cite{Arr}, that the solutions $A(t)$ to \eqref{BFv02} and the corresponding SCF-solutions $(\omega(t),g(t))$ on the solvable Lie group $G_{\mu_{A_0}}$ satisfy the following properties:

\begin{itemize}
\item $A(t)$ and hence the SCF-solution $(\omega(t),g(t))$ are defined for $t\in(T_-,\infty)$ since $|A(t)|$ is strictly decreasing unless $A(t)\equiv A_0$ (i.e. $A_0^t=-A_0$).  (See \cite[Proposition 3.4]{Arr}).

\item[ ]
\item The (scaling invariant) quantity
$$
\frac{|[A,A^t]|^2}{|A|^4}
$$
is strictly decreasing along the flow, unless $\mu_{A_0}$ is an algebraic SCF-soliton (see Theorem \ref{SA-sol}, (ii)).  This implies that any limit $B=\lim\limits_{t_k\to\infty}\frac{A(t_k)}{|A(t_k)|}$ gives rise to an algebraic SCF-soliton $\mu_B$.  (See \cite[Lemma 3.6 and Corollary 3.7]{Arr}).

\item[ ]
\item  There is always a subsequence $t_k\to\infty$ such that if $c_k:=|A(t_k)|^{-1}$, then the almost-K\"ahler manifolds $\left(G,\tfrac{1}{c_k^2}\omega(t_k),\tfrac{1}{c_k^2}g(t_k)\right)$ converge in the pointed sense to an algebraic SCF-soliton $(G_{\mu_B},\omega_0,g_0)$, as $k\to\infty$, where $B$ is any accumulation point of $\{ A(t)/|A(t)|:t\in[0,\infty)\}$ (see Corollary \ref{conv}).

\item[ ]
\item If $\tr{A_0^2}\geq 0$, then the SCF-solution $(\omega(t),g(t))$ is type-III, in the sense that there is a constant $C>0$ (which in this particular case depends only on the dimension $n$ but in general it may depend on the solution) such that
$$
|R(\omega(t),g(t))|+|\Riem(g(t))|\leq \frac{C}{t}, \qquad\forall t\in(0,\infty),
$$
where $R$ and $\Riem$ respectively denote the curvature tensors of the Chern and the Levi-Civita connections.  (See \cite[Proposition 3.14]{Arr}; recall that we also have that $\ddt\tr{S(A)^2}\leq\left(\tr{S(A)^2}\right)^2$).

\item[ ]
\item The Chern scalar curvature $\tr{P}=-2a^2$ and the scalar curvature $R=-a^2-\tr{S(A)^2}$ are both increasing and go to $0$ as $t\to\infty$.

\item[ ]
\item In the unimodular case (i.e. $a=0$), $\frac{A(t)}{|A(t)|}$ converges, as $t\to\infty$, to a matrix $B$ such that $\mu_B$ is an algebraic SCF-soliton. (See \cite[Lemma 4.1]{Arr}).
\end{itemize}

{\small
\begin{table}
\renewcommand{\arraystretch}{1.6}
$$
\begin{array}{|c|c|c|c|}\hline
  \ggo & \mbox{Lie bracket} & 2-\mbox{form}  & constraint    \\
\hline\hline

\RR^4 &(0,0,0,0) & \omega= e^1\wedge e^2+ e^3\wedge e^4 &  -  \\ \hline

\rg \hg_3 &(0,0,12,0) & \omega= e^1\wedge e^4+ e^2\wedge e^3 &  -  \\ \hline

\rg \rg_{3,0} & (0,12,0,0) & \omega= e^1\wedge e^2+ e^3\wedge e^4 & - \\ \hline

\rg \rg_{3,-1} & (0,12,-13,0) & \omega = e^1\wedge e^4+ e^2\wedge e^3 &  - \\ \hline

\rg \rg'_{3,0} & (0,13,-12,0) & \omega= e^1\wedge e^4+ e^2\wedge e^3 & -  \\ \hline

\rg_2 \rg_2 & (0,12,0,34) & \omega_\alpha = e^1\wedge e^2+ \alpha \, e^1\wedge e^3+ e^3\wedge e^4 & \alpha \ge 0 \\ \hline

 \rg_2' & (0,0,13+24,14-\tfrac{5}{3}\cdot 23) & \omega= e^1\wedge e^3+ e^2\wedge e^4 & - \\ \hline

 \ngo_4 & (0,41,42,0) & \omega= e^1\wedge e^2 + e^3\wedge e^4 &  - \\ \hline

\rg_{4,0} &  (41,43,0,0) & \omega_{\pm} = e^1\wedge e^4 \pm e^2\wedge e^3  & -\\ \hline

\rg_{4,-1} &  (41,43-42,-43,0) & \omega=e^1\wedge e^3 + e^2\wedge e^4  & -\\ \hline

\rg_{4,-1,\lambda} &  (41,-42,\lambda \cdot 43,0) & \omega=e^1\wedge e^2 + e^3\wedge e^4  & -1 \le \lambda <1 \\ \hline

\rg_{4,\lambda, -\lambda} &   (41,\lambda\cdot 42,-\lambda\cdot 43,0) & \omega=e^1\wedge e^4 + e^2\wedge e^3  & -1 < \lambda < 0 \\ \hline

\rg'_{4,0, \lambda} &   (41,\lambda \cdot43,-\lambda\cdot 42,0) & \omega_\pm = e^1\wedge e^4 \pm e^2\wedge e^3  & 0 < \lambda  \\ \hline

\multirow{2}{*}{$\dg_{4,1}$} & \multirow{2}{*}{$(41,0,12+43,0)$} & \omega_1=e^1\wedge e^2 - e^3\wedge e^4  &  -\\ \cline{3-4}
 & & \omega_2 = e^1\wedge e^4 + e^2\wedge e^3 & -\\ \hline

\multirow{2}{*}{$\dg_{4,2}$} & (2\cdot41,-42,12+43,0) & \omega_1=e^1\wedge e^2 - e^3\wedge e^4  &  -\\ \cline{2-4}
 & (2\cdot41,-42,2 \cdot12+43,0) & \omega_\pm = e^1\wedge e^4 \pm e^2\wedge e^3  &  -\\ \hline

\dg_{4,\lambda} & (\lambda\cdot 41,(1-\lambda)\cdot42,12+43,0) & \omega=e^1\wedge e^2 - e^3\wedge e^4  &  \lambda \ge \frac{1}{2}, \lambda \ne 1,2 \\ \hline

\dg'_{4,\lambda} & \begin{array}{c} \left(\tfrac{\sqrt{\lambda}}{2}\cdot41 + \frac{1}{\sqrt{\lambda}}\cdot42,\tfrac{\sqrt{\lambda}}{2}\cdot 42 - \;\;\right. \\
\;\;\left. \tfrac{1}{\sqrt{\lambda}}\cdot 41, \sqrt{\lambda}\cdot 12 + \sqrt{\lambda}\cdot 43,0 \right)\end{array}  & \omega_\pm = \pm(e^1\wedge e^2 - e^3\wedge e^4)  &  \lambda > 0  \\ \hline

\hg_4 &   (\frac{1}{2} \cdot 41+42,\frac{1}{2}\cdot 42,12+43,0) & \omega_\pm = \pm (e^1\wedge e^2 - e^3\wedge e^4)  & -\\ \hline
\end{array}
$$
\caption{Classification of $4$-dimensional symplectic Lie algebras \cite{Ovn}}\label{n1}
\end{table}}

\subsection{Compact quotients}\label{lattices}
The Lie group $G_{\mu_A}$ admits a lattice (i.e. a cocompact discrete sugbroup) if and only if
$$
\sigma e^{\alpha A}\sigma^{-1}\in\Sl_{2n-1}(\ZZ),
$$
for some nonzero $\alpha\in\RR$ and $\sigma\in\Gl_{2n-1}(\RR)$ (see \cite[Section 4]{Bck}).  In that case, a lattice is given by
$$
\Gamma=\exp\left(\sigma^{-1}\ZZ^{2n-1}\rtimes\ZZ\alpha e_{2n}\right).
$$
Moreover, if $\Spec(A)\subset\RR$ (i.e. $\mu_A$ is completely solvable), then two of these lattices differ by an automorphism of $G_{\mu_A}$ if and only if
$\sigma e^{\alpha A}\sigma^{-1}$ is conjugate to $\tau e^{\beta A}\tau^{-1}$ or its inverse in $\Gl_{2n-1}(\ZZ)$ (see \cite[Theorem 2.5]{Hng}).  We refer to \cite{CnsMcr} for a complete study of lattices on $6$-dimensional almost abelian groups, including results on formality and half-flatness of invariant and non-invariant symplectic structures on the corresponding compact quotients.

We have found in Section \ref{dim4-sec} a (strongly algebraic) SCF-soliton on any symplectic structure on unimodular Lie groups of dimension $4$, thus showing that any $4$-dimensional compact solvmanifold  $G/\Gamma$ does admit a SCF-soliton.  The next example shows that this is no longer true in dimension $6$.

\begin{example}\label{6latt}
By setting $a=0$, $v=0$ and
$$
A_1= \left[\begin{array}{cc|cc} 0&0&1&0\\ 0&\log{\lambda}&0&0 \\ \hline 0&0&0&0\\
0&0&0&-\log{\lambda}\end{array}\right]\in\spg(2,\RR), \qquad\lambda=\tfrac{3+\sqrt{5}}{2},
$$
we obtain a symplectic Lie group $(G_{\mu_A},\omega)$ which does not admit any algebraic SCF-soliton, as $A$ is neither nilpotent nor semisimple (see Corollary \ref{SA-cor}).  On the other hand, there exists $\sigma\in\Gl_4(\RR)$ such that
$$
\sigma e^{A}\sigma^{-1} =  \left[\begin{matrix} 1&0&1&0\\ 0&2&0&1\\
0&0&1&0\\ 0&1&0&1\end{matrix}\right]\in\Sl_4(\ZZ),
$$
and so $\Gamma=\exp\left(\sigma^{-1}\ZZ^4\rtimes\ZZ e_6\right)$ is a
lattice of $G_{\mu_A}$.
\end{example}

Concerning the SCF-solution starting at the almost-K\"ahler structure $(G_{\mu_A},\omega,g)$ in the example above, it is straightforward to prove that the family
$$
A_1= \left[\begin{array}{cc|cc} 0&0&b&0\\ 0&a&0&0\\\hline 0&0&0&0\\
0&0&0&-a\end{array}\right]\in\spg(2,\RR), \qquad a,b\in\RR,
$$
is invariant for the bracket flow equation (see \eqref{BFv0})
$$
A_1'=  -\unm\tr{S(A_1)^2})A_1+\unm[A_1,[A_1,A_1^t]],
$$
which becomes the following ODE system for the variables $a(t),b(t)$:
$$
\left\{\begin{array}{l}
a'= -(a^2+\unc b^2)a, \\ \\
b'= -(a^2+\tfrac{5}{4}b^2)b.
\end{array}\right.
$$
By a standard qualitative analysis, we obtain long-time existence (i.e. $T_+=\infty$) for all these SCF-solutions and that $(a,b)\to(0,0)$, as $t\to\infty$, from which follows that $(G_{\mu_A},\omega(t),g(t))$, with $A$ as in Example \ref{6latt}, converges to the euclidean space $(\RR^6,\omega_0,g_0)$ in the pointed sense, as $t\to \infty$.  Note that $P\equiv 0$ and the scalar curvature $R=-\tr{S(A_1)^2}$ is strictly increasing and converges to $0$ as $t\to\infty$.

Furthermore,
$$
\lim\limits_{t\to\infty} A(t)/|A(t)|= B:= \tfrac{1}{\sqrt{2}}\left[\begin{array}{cc|cc} 0&0&0&0\\ 0&1&0&0\\\hline 0&0&0&0\\
0&0&0&-1\end{array}\right]\in\spg(2,\RR),
$$
and thus pointed convergence of $(G_{\mu_A},c(t)\omega(t),c(t)g(t))$ toward the (strongly algebraic) SCF-soliton $(G_{\mu_B},\omega_0,g_0)$ (see Theorem \ref{SA-sol}, (ii)) follows for $c(t)=|{A(t)}|^2$ (see Corollary \ref{conv}), which is isometric to $\rg\rg_{3,-1}\times\RR^2$, where $\rg\rg_{3,-1}$ is the SCF-soliton given in Table \ref{n3}.

\begin{remark}
We note that $G_{\mu_B}$ also admits a lattice, say $\Lambda$.  It would be very useful to understand what kind of convergence one obtains for the sequence of compact almost-K\"ahler manifolds $(G_{\mu_A}/\Gamma,c(t)\omega(t),c(t)g(t))$ toward $(G_{\mu_B}/\Lambda,\omega_0,g_0)$, as $t\to\infty$.  Notice that $G_{\mu_B}/\Lambda$ is compact and not homeomorphic to $G_{\mu_A}/\Gamma$, thus pointed convergence can not hold for any subsequence.  The diameters of $(G_{\mu_A}/\Gamma,g(t))$ might go to infinity, in which case only pointed Gromov-Hausdorff convergence may be expected.
\end{remark}

{\footnotesize
\begin{table}
$$
\renewcommand{\arraystretch}{2.3}
\begin{array}{|c|c|c|c|c|c|c|}\hline
 \multirow{2}{*}{$\ggo$} & \multirow{2}{*}{$\omega$} & \multicolumn{2}{|c|}{P}  & \multicolumn{2}{|c|}{\Ricac}  & \multirow{2}{*}{Obs.}  \\ \cline{3-6}
 & &   c_1 & D_1 & c_2 & D_2 & \\
\hline \hline

\RR^4 & e^1\wedge e^2+ e^3\wedge e^4 & c_1 & -c_1I & c_2 & -c_2I &  \mbox{flat} \\ \hline

\rg \hg_3 & e^1\wedge e^4+ e^2\wedge e^3 & 0 & 0 & -\tfrac{5}{4} & (1,\tfrac{3}{4},\tfrac{7}{4},\tfrac{3}{2}) &  - \\ \hline

\rg \rg_{3,0} & e^1\wedge e^2+ e^3\wedge e^4 & -1& (0,0,1,1)& 0 & 0 & \mbox{K}\\ \hline

\rg \rg_{3,-1} & e^1\wedge e^4+ e^2\wedge e^3 & 0 & 0& -1& (0,1,1,2) & - \\ \hline

\rg \rg'_{3,0} & e^1\wedge e^4+ e^2\wedge e^3& 0 &0 & 0 & 0 & \mbox{flat}  \\ \hline

\rg_2\rg_2 & \omega_0=e^1\wedge e^2+ e^3\wedge e^4 & -1 &0 & 0 & 0 & \mbox{K-E}  \\ \hline

\rg_2' & e^1\wedge e^3+ e^2\wedge e^4 & -\tfrac{2}{3} &0 & \tfrac{4}{9} & (0,0,-\tfrac{8}{9},-\tfrac{8}{9}) & -  \\ \hline

\ngo_4 & e^1\wedge e^2 + e^3\wedge e^4 & 0 & \left[\begin{smallmatrix} &&0&0\\ &&0&-\unm\\ -\unm&0&&\\0&0&&\end{smallmatrix}\right] & -\frac{5}{4} & (1,\tfrac{3}{2},2,\unm) & -\\ \hline

\rg_{4,-1,\lambda} & e^1\wedge e^2 + e^3\wedge e^4  & - \lambda^2 & (\lambda^2,\lambda^2,0,0) & -(1+\lambda^2) & A_\lambda & - \\ \hline

\rg_{4,\lambda, -\lambda} & e^1\wedge e^4 + e^2\wedge e^3  & -1& (0,1,1,0)  &  -\lambda^2& B_\lambda  & -\\ \hline

\rg'_{4,0, \lambda} &  e^1\wedge e^4 \pm e^2\wedge e^3  & -1 & (0,1,1,0) &  0 & 0 & K \\ \hline

\multirow{2}{*}{$\dg_{4,1}$} & e^1\wedge e^2 - e^3\wedge e^4 & -\frac{3}{2} & 0 & \multirow{2}{*}{$-\unc$}&\multirow{2}{*}{$(-\tfrac{3}{4},\frac{5}{4},\unm,0)$} & - \\ \cline{2-4}
& e^1\wedge e^4 + e^2\wedge e^3 & -2 & (0,2,2,0) &  &  & - \\ \hline

\multirow{3}{*}{$\dg_{4,2}$} & e^1\wedge e^2 - e^3\wedge e^4 &  -\frac{3}{2} & 0 & -\frac{9}{4}& (-\frac{3}{4},\frac{21}{4},\frac{9}{2},0) & - \\ \cline{2-7}

  & e^1\wedge e^4 + e^2\wedge e^3  & -6 & (0,6,6,0) & \multirow{2}{*}{$0$} & \multirow{2}{*}{$0$} & \mbox{K}\\ \cline{2-4}\cline{7-7}

   & e^1\wedge e^4 - e^2\wedge e^3  & -2 & (0,2,2,0) &  &  & - \\ \hline

\dg_{4,\lambda} & e^1\wedge e^2 - e^3\wedge e^4 &  -\frac{3}{2} & 0 &  -(\lambda-\unm)^2& C_\lambda & \mbox{K-E}\, (\lambda= \tfrac{1}{2}) \\ \hline

\dg'_{4,\lambda} & \pm(e^1\wedge e^2 - e^3\wedge e^4)  &  -\frac{3}{2} & 0 & 0 & 0 & \mbox{K-E} \\ \hline
\end{array}
$$
\caption{SCF-solitons in dimension $4$}\label{n3}
\end{table}}

\section{SCF-solitons in dimension $4$}\label{dim4-sec}

We now study the existence problem for SCF-solitons on $4$-dimensional Lie groups.  We have listed in Table \ref{n1} all the symplectic structures up to isomorphism on $4$-dimensional Lie algebras according to the classification obtained by Ovando in \cite{Ovn}.  We have changed the basis $\{ e_i\}$ used in \cite{Ovn} in only three cases: for $\rg_2'$ we took $\{ e_1,\sqrt{\tfrac{5}{3}}e_2,-\sqrt{\tfrac{5}{3}}e_3,e_4\}$ instead, for $\omega_\pm$ on $\dg_{4,2}$ we used $\{ e_1,\sqrt{2}e_2,\frac{1}{\sqrt{2}}e_3,e_4\}$, and for $\omega_\pm$ on $\dg_{4,\lambda}'$, our basis is $\{ e_1,e_2,\frac{1}{\sqrt{\lambda}}e_3,\frac{1}{\sqrt{\lambda}}e_4\}$.  The notation we have used in Table \ref{n1} for Lie brackets can be understood from the example of $\hg_4$ in the last line, whose Lie bracket is described as $(\frac{1}{2} \cdot 41+42,\frac{1}{2}\cdot 42,12+43,0)$ and means
$$
[e_4,e_1]=\unm e_1, \quad [e_4,e_2]=e_1+\unm e_2, \quad [e_4,e_3]=e_3, \quad [e_1,e_2]=e_3.
$$
We have found a strongly algebraic SCF-soliton on each symplectic structure on a $4$-dimensional Lie group, with the exception of the following four cases:
$$
(\rg_2\rg_2,\omega_\alpha), \quad \alpha>0,  \quad (\rg_{4,0},\omega_\pm), \quad (\rg_{4,-1},\omega), \quad (\hg_4,\omega_\pm).
$$
We were able to prove the non-existence of an algebraic SCF-soliton only in the cases of $(\rg_{4,0},\omega_\pm)$ and $(\rg_{4,-1},\omega)$ (see Lemma \ref{r40-r4-1}).  The SCF-soliton almost-K\"ahler structures and their respective Chern-Ricci and Ricci operators are given in Table \ref{n3} as diagonal matrices with respect to the orthonormal basis $\{ e_1,e_2,e_3,e_4\}$ (except $\ngo_4$), together with the constants $c_i$ and the derivations $D_i$ such that $P=c_1I+D_1$ and $\Ricac=c_2I+D_2$.  We note that they are all expanding SCF-solitons since $c=c_1+c_2<0$, with the only exception of the flat structure $\rg\rg'_{3,0}$.  Most of these solitons were obtained by either direct computation or by using the structure results for almost abelian solvmanifolds given in Theorem \ref{SA-sol}, with the exception of $\rg'_2$, where the LSA construction considered in Section \ref{exmirror} was crucial.

In the last column we specify when the metric is K\"ahler-Einstein (K-E), only K\"ahler (K) or flat (i.e. isometric to $\RR^4$).  Recall that such structures are all K\"ahler-Ricci solitons.

In some cases, in order to simplify the description of the derivations in Table \ref{n3}, we have introduced the following notation:
$$
\begin{array}{c}
A_\lambda:=(1+\lambda^2 -\lambda, 1+\lambda^2 +\lambda,2(1+\lambda^2),0), \quad
B_\lambda:=(2\lambda^2,\lambda^2 -\lambda,\lambda^2 +\lambda,0), \\ \\
C_\lambda:=(\lambda^2-3\lambda+\tfrac{5}{4},\lambda^2+\lambda-\tfrac{3}{4},2(\lambda^2-\lambda)+\unm,0).
\end{array}
$$

\begin{remark}
A SCF-soliton $(G,\omega,g)$ in Table \ref{n3} is {\it static} (i.e. $p=c\omega$ and $\ricac=0$, or equivalently, its SCF-evolution is $(\omega(t),g(t))=(-2ct+1)(\omega,g)$) if and only if it is K\"ahler-Einstein.  This has been proved for any compact static almost-K\"ahler structure of dimension $4$ in \cite[Corollary 9.5]{StrTn2}.
\end{remark}

\subsection{Compact symplectic surfaces}\label{surfaces}
It follows from the classification given in Table \ref{n1} that there are exactly five (simply connected) solvable Lie groups of dimension $4$ admitting a left-invariant symplectic structure which also admit a lattice (i.e. compact discrete subgroup), giving rise to the compact symplectic surfaces which are solvmanifolds.  Their Lie algebras are:
$\RR^4$ (Complex tori), $\rg\hg_3$ (Primary Kodaira surfaces), $\rg\rg_{3,-1}$, $\rg\rg'_{3,0}$ (Hyperelliptic surfaces) and $\ngo_4$.  We refer to \cite{Hsg} for a comparison with compact complex surfaces which are solvmanifolds.  Recall that $\rg\rg_{3,-1}$ and $\ngo_4$ do not admit invariant complex structures.

According to Table \ref{n3}, they all admit a SCF-soliton which is steady in the flat cases $\RR^4$ and $\rg\rg'_{3,0}$ and expanding in the other three cases.

Since each of these five Lie algebras admits a codimension one abelian ideal, it follows from Section \ref{muA-dim4-sec} that any left-invariant almost-K\"ahler structure on them is equivalent to $(G_{\mu_A},\omega,g)$ for some
\begin{equation}\label{Adim4-surf}
A= \left[\begin{matrix}0 & b & c \\  0 & d & e \\ 0 & f &-d \end{matrix}\right], \qquad b,c,d,e,f  \in \RR.
\end{equation}
It is easy to check  that
$$
\mu_A\simeq\left\{\begin{array}{lcl}
\RR^4 && A=0; \\
\rg\hg_3 && d^2+ef=0, \quad db+fc=0, \quad eb-dc=0, \quad A\ne 0; \\
\rg\rg_{3,-1} && d^2+ef>0; \\
\rg\rg'_{3,0} && d^2+ef<0; \\
\ngo_4 && d^2+ef=0, \quad (db+fc,eb-dc)\ne (0,0),
\end{array} \right.
$$
and the Chern-Ricci and Ricci operators can be computed by using \eqref{PA} and \eqref{RicacA}, respectively:
$$
 P = \left[\begin{smallmatrix}
 0& - \frac{db +fc}{2} \quad & - \frac{eb-dc}{2}   & 0\\ & & & \\
 &  0 & 0 &  \frac{eb-dc}{2} \\ & & & \\
 &  & 0 & - \frac{db +fc}{2}\\  & & & \\
 &  &   &0
\end{smallmatrix}\right],
$$

$$
\Ricac  = \left[\begin{smallmatrix}
 d^2 + \frac{b^2+c^2}{2} + \frac{(e + f)^2}{4} & \frac{db+ce}{4} & \frac{bf-dc}{4} & 0 \\
 \frac{db+ce}{4} & \frac{e^2-f^2}{2} -\frac{b^2 -c^2}{4}  & \quad d(f-e)- \frac{bc}{2} & \frac{bf-dc}{4} \\
 \frac{bf-dc}{4}  &  d(f-e)- \frac{bc}{2} & \frac{f^2-e^2}{2} +\frac{b^2 -c^2}{4}   &-\frac{db+ce}{4}  \\
0 & \frac{bf-dc}{4} & -\frac{db+ce}{4}  & - d^2 - \frac{b^2+c^2}{2} - \frac{(e + f)^2}{4}
\end{smallmatrix}\right].
$$
Each of the following five matrices $A$ provides a SCF-soliton on the corresponding Lie group in the order we are using:
$$
\left[\begin{matrix}0 & 0 & 0 \\  0 & 0 & 0 \\ 0 & 0 & 0 \end{matrix}\right], \quad
\left[\begin{matrix}0 & 0 & 0 \\  0 & 0 & 1 \\ 0 & 0 & 0 \end{matrix}\right], \quad
\left[\begin{matrix}0 & 0 & 0 \\  0 & 1 & 0 \\ 0 & 0 &-1 \end{matrix}\right], \quad
\left[\begin{matrix}0 & 0 & 0 \\  0 & 0 & -1 \\ 0 & 1 & 0 \end{matrix}\right], \quad
\left[\begin{matrix}0 & 1 & 0 \\  0 & 0 & 1 \\ 0 & 0 & 0 \end{matrix}\right].
$$
As an application of Section \ref{BFmuA-sec}, for each starting almost-K\"ahler structure $(G_{\mu_A},\omega,g)$ with $A$ as in \eqref{Adim4-surf}, we obtain that $A(t)/|A(t)|$ converges to one of the soliton matrices $B$ above such that $G_{\mu_A}$ and $G_{\mu_B}$ are isomorphic, that is, the one with same eigenvalues as $A$ up to scaling.  Thus pointed convergence of $(G_{\mu_A},c(t)\omega(t),c(t)g(t))$ toward the (strongly algebraic) SCF-soliton $(G_{\mu_B},\omega,g)$ follows for $c(t)=|{A(t)}|^2$ (see Corollary \ref{conv}).

\begin{remark}
It would be interesting to know if this gives rise to (pointed) Gromov-Hausdorff convergence for the corresponding compact quotients.
\end{remark}

\section{LSA construction}\label{weak}

All SCF-solitons we have found in Sections \ref{muA-sec} and \ref{dim4-sec} are on solvable Lie groups and moreover, they are all expanding in the nonflat case (see Remark \ref{allexp} and Table \ref{n3}).  For the Ricci flow, it is well known that any shrinking homogeneous Ricci soliton is trivial, in the sense that it is finitely covered by a product of a compact Einstein homogeneous manifold with a euclidean space (see \cite{PtrWyl}), and any steady homogeneous Ricci soliton is necessarily flat.  However, it is an open question whether any expanding homogeneous Ricci soliton is isometric to a left-invariant metric on a solvable Lie group, which is now known to be essentially equivalent to Alekseevskii's Conjecture (see e.g. \cite{alek,ArrLfn,JblPtr} and the references therein).

In this section, in order to search for SCF-solitons beyond the solvable case, we shall study a construction attaching to each $n$-dimensional left-symmetric algebra an almost-K\"ahler structure on a $2n$-dimensional Lie group (see e.g. \cite{Bym,AndSlm,Ovn} for further information on this construction).  Our search succeeded in finding a shrinking SCF-soliton on the Lie algebra $\ug(2)\ltimes\HH$ (see Example \ref{u(2)}) and an expanding SCF-soliton on $\glg_2(\RR)\ltimes\RR^4$ (see Example \ref{gl2nice}).

A {\it left-symmetric algebra} (LSA for short) structure on a vector space $\ggo$ is a bilinear product $\cdot:\ggo\times\ggo\longrightarrow\ggo$ satisfying the condition
\begin{equation}\label{LSA-def}
X\cdot(Y\cdot Z)-(X\cdot Y)\cdot Z = Y\cdot (X\cdot Z)-(Y\cdot X)\cdot Z, \qquad\forall X,Y,Z\in\ggo.
\end{equation}
(From now on, the phrase `for all $X,Y,Z\in\ggo$' will be understood in any formula containing $X,Y,Z$).  This implies that
\begin{equation}\label{lb-lsa}
[X,Y]_\ggo:=X\cdot Y-Y\cdot X,
\end{equation}
is a Lie bracket on $\ggo$ and if $L(X):\ggo\longrightarrow\ggo$ denotes LSA left-multiplication by $X$ (i.e. $L(X)Y=X\cdot Y$), then $L$ is a representation:
$$
L([X,Y]_\ggo)=L(X)L(Y)-L(Y)L(X).
$$
We now show how each LSA structure on $\ggo$ determines an almost-K\"ahler structure on $\ggo\oplus\ggo$.  Consider the representation $\theta:\ggo\longrightarrow\End(\ggo)$ given by
\begin{equation}\label{rep-lsa}
\theta(X):=-L(X)^t,
\end{equation}
where $L(X)^t$ denotes the transpose of the map $L(X)$ with respect to an inner product $\ip$ on $\ggo$, which will be considered fixed from now on, and define the Lie algebra $\ggo\ltimes_\theta\ggo$ with Lie bracket
\begin{equation}\label{m-lb}
[(X,Y),(Z,W)]:=\left( [X,Z]_\ggo, \theta(X)W-\theta(Z)Y \right).
\end{equation}
Note that by \eqref{lb-lsa} and \eqref{rep-lsa}, $\lb_\ggo$ is determined by $\theta$ as follows,
\begin{equation}\label{oc}
[X,Y]_\ggo=-\theta(X)^tY+\theta(Y)^tX.
\end{equation}
Consider also the almost-complex structure $J:\ggo\oplus\ggo\longrightarrow\ggo\oplus\ggo$ defined by
$$
J(X,Y):=(Y,-X), \qquad \mbox{i.e.} \quad J=\left[\begin{array}{cc} 0&I\\ -I&0\end{array}\right].
$$

On the right we are writing $J$ as a matrix in terms of the basis $\{(e_i,0)\}\cup\{(0,e_i)\}$, where $\{ e_i\}$ is any orthonormal basis of $\ggo$.  Such basis of $\ggo\oplus\ggo$ will be fixed and used without any further mention, e.g. to write operators as matrices.  A $2$-form $\omega$ on $\ggo\oplus\ggo$ can therefore be defined by
$$
\omega:=g(J\cdot,\cdot),  \qquad\mbox{where}\qquad g:=\ip\oplus\ip,
$$
or equivalently,
$$
\omega=-\sum_{i=1}^n (e^i,0)\wedge (0,e^i),
$$
where $\{ e^i\}$ denotes the dual basis of $\{ e_i\}$.

The almost-hermitian Lie algebra $(\ggo\ltimes_\theta\ggo,\omega,g)$ is therefore completely determined by the LSA structure, as $\theta$ and $\lb_\ggo$ are so and the whole `linear algebra' data (i.e. $(\ggo\oplus\ggo,\omega,g)$) has been fixed.  Moreover, it is easy to see that condition \eqref{oc} is equivalent to $d\omega=0$.  Summing up,

\begin{proposition}\label{LSA-aK}
Any LSA structure on $\ggo$ defines an almost-K\"ahler Lie algebra
$$
(\ggo\ltimes_\theta\ggo,\omega,g).
$$
\end{proposition}

\begin{remark}\label{mirror}
If we define a Lie bracket $\lb^*$ on $\ggo\oplus\ggo$ as in \eqref{m-lb} by using the same $\lb_{\ggo}$ but the representation $\theta^*(X)=L(X)=-\theta(X)^t$ instead of $\theta$, then what we obtain is a hermitian Lie algebra
$$
(\ggo\ltimes_{\theta^*}\ggo,J,g),
$$
i.e. $J$ is integrable.  Together, the corresponding complex manifold $(G_{\theta^*},J)$ and the symplectic manifold $(G_\theta,\omega)$ form a {\it weak mirror pair}, i.e. their associated differential
Gerstenhaber algebras are quasi-isomorphic (see e.g. \cite{mirror}).
\end{remark}

\begin{remark}
The left-invariant affine connection on the corresponding Lie group $\nabla:\ggo\times\ggo\longrightarrow\ggo$ defined by
$$
\nabla_XY:= X\cdot Y = -\theta(X)^tY,
$$
is flat (i.e. $\nabla_{[X,Y]_\ggo}=[\nabla_X,\nabla_Y]$) and torsion free (i.e. $[X,Y]_\ggo=\nabla_XY-\nabla_YX$).
\end{remark}

\begin{remark}\label{assume}
We will assume in what follows that $(0,\ggo)$ is invariant by any element of $\Aut(\ggo\ltimes_\theta\ggo)$ for all the LSA structures considered.  This for example holds when the abelian ideal $(0,\ggo)$ is the nilradical of $\ggo\ltimes_\theta\ggo$.
\end{remark}

\begin{proposition}\label{LSA-iso}
Two symplectic Lie algebras $(\ggo\ltimes_{\theta_1}\ggo,\omega)$ and $(\ggo\ltimes_{\theta_2}\ggo,\omega)$ are isomorphic if and only if there exists $\psi\in\Gl(\ggo)$ such that
\begin{equation}\label{psi}
L_2(\psi X)=\psi L_1(X)\psi^{-1}, \qquad\forall X\in\ggo,
\end{equation}
i.e. the corresponding LSA structures are isomorphic.
\end{proposition}

\begin{proof}
If \eqref{psi} holds, then it is easy to check that $\vp=\left[\begin{matrix} \psi&0\\ 0&(\psi^t)^{-1}\end{matrix}\right]$ is a Lie algebra isomorphism between $\ggo\ltimes_{\theta_1}\ggo$ and $\ggo\ltimes_{\theta_2}\ggo$.  Since $\vp\in\Spe(\omega)$, we obtain that the symplectic Lie algebras are also isomorphic.

Conversely, due to our assumption (see Remark \ref{assume}), any isomorphism between the Lie algebras has the form $\vp=\left[\begin{matrix} \vp_1&0\\ \vp_3&\vp_2\end{matrix}\right]$, which implies that $\vp_1[\cdot,\cdot]_{\ggo_1}=[\vp_1\cdot,\vp_1\cdot]_{\ggo_2}$ and $\theta_2(\vp_1X)=\vp_2\theta_1(X)\vp_2^{-1}$.  But since $\vp\in\Spe(\omega)$ we have that $\vp_2=(\vp_1^t)^{-1}$, from which condition \eqref{psi} easily follows for $\psi=\vp_1$.
\end{proof}

In much the same way, we obtain the following criterium for equivalence.

\begin{proposition}\label{LSA-equiv}
Two almost-K\"ahler structures $(\ggo\ltimes_{\theta_1}\ggo,\omega,g)$ and $(\ggo\ltimes_{\theta_2}\ggo,\omega,g)$ are equivalent if and only if there exists an orthogonal map $\psi\in\Or(\ggo,\ip)$ such that
\begin{equation}\label{psiort}
L_2(\psi X)=\psi L_1(X)\psi^{-1}, \qquad\forall X\in\ggo.
\end{equation}
\end{proposition}

\begin{example}\label{gl2Bii}
Consider on $\ggo=\glg_2(\RR)$ the basis
$$
e_1=\left[\begin{smallmatrix} 0&1\\ 0&0\end{smallmatrix}\right], \quad e_2=\left[\begin{smallmatrix} 0&0\\ 1&0\end{smallmatrix}\right], \quad e_3=\left[\begin{smallmatrix} 1&0\\ 0&-1\end{smallmatrix}\right], \quad e_4=\left[\begin{smallmatrix} 1&0\\ 0&1\end{smallmatrix}\right],
$$
whose Lie bracket relations are
$$
[e_1,e_2]=e_3, \quad [e_3,e_1]=2e_1, \quad [e_3,e_2]=-2e_2,
$$
and the one-parameter family of LSA structures defined for any $\alpha\geq 0$ by
$$
\begin{array}{l}
L_\alpha(e_1)=\left[\begin{smallmatrix} 0&0&-1&1+\alpha\\ 0&0&0&0 \\ 0&(1+\alpha)/2&0&0\\ 0&1/2&0&0\end{smallmatrix}\right], \quad L_\alpha(e_2)=\left[\begin{smallmatrix} 0&0&0&0\\ 0&0&1&1-\alpha \\ -(1-\alpha)/2&0&0&0\\ 1/2&0&0&0\end{smallmatrix}\right], \\ \\
L_\alpha(e_3)=\left[\begin{smallmatrix} 1&0&0&0\\ 0&-1&0&0 \\ 0&0&\alpha&1-\alpha^2\\ 0&0&1&-\alpha\end{smallmatrix}\right],
\quad L_\alpha(e_4)=\left[\begin{smallmatrix} 1+\alpha&0&0&0\\ 0&1-\alpha&0&0 \\ 0&0&1-\alpha^2&-\alpha(1-\alpha^2)\\ 0&0&-\alpha&1+\alpha^2\end{smallmatrix}\right].
\end{array}
$$
It is proved in \cite{Brd} that these LSA structures are pairwise non-isomorphic and henceforth, according to Proposition \ref{LSA-iso}, $(\ggo\ltimes_{\theta_\alpha}\ggo,\omega)$ is a family of pairwise non-isomorphic symplectic Lie algebras.  Actually, the Lie algebras $\ggo\ltimes_{\theta_\alpha}\ggo$, $\alpha\geq 0$ are pairwise non-isomorphic, as it is easy to check that the spectrum of $L_\alpha(e_4)$ is an invariant and equals $\{ 1\pm\alpha,1\pm\alpha\}$.  Notice that $\alpha=0$ corresponds to the usual multiplication of matrices in $\glg_2(\RR)$, and is the only one associative among the family.  In order to obtain the complete classification of LSA structures on $\glg_2(\RR)$ up to isomorphism, an extra one-parameter family and two more (isolated) structures must be added (see \cite[Theorem 3]{Brd} and \cite[Section 5.1]{Bas}).
\end{example}

\subsection{Chern-Ricci and Ricci curvature}
We compute in this section the Chern-Ricci operator $P$ and the anti-J-invariant Ricci operator $\Ricac$ for the almost-K\"ahler structure $(\ggo\ltimes_\theta\ggo,\omega,g)$ from Proposition \ref{LSA-aK}.

We first define $A,A^*\in\ggo$ by
$$
A:=\sum_{i=1}^n\theta(e_i)e_i, \qquad A^*:=-\sum_{i=1}^n\theta(e_i)^te_i =\sum_{i=1}^n e_i\cdot e_i.
$$
By a straightforward computation, one obtains that the Chern-Ricci form $p$ vanishes on both $\ggo$-summands and
\begin{equation}\label{m-CRform}
p((X,0),(0,Y))=-\unm\la\theta(X)Y,A^*\ra+\unm\tr{\ad_\ggo{\theta(X)Y}}+\unm\tr{\theta(\theta(X)Y)}.
\end{equation}
The Chern-Ricci operator $P$ therefore leaves invariant each $\ggo$-summand.  More precisely,

\begin{lemma}  $P=\left[\begin{array}{cc} P&0\\ 0&P^t\end{array}\right]$, where $P\in\End(\ggo)$ is defined by
\begin{equation}\label{Psymp}
P=\unm\ad_\ggo{(A^*-A)}+\unm\theta(A^*-A)^t.
\end{equation}
\end{lemma}

\begin{remark}\label{PRZ}
If $Z:=\unm(A^*-A)$, then $P\in\End(\ggo\oplus\ggo)$ satisfies $P=\ad{Z}+(\ad{Z})^{t_\omega}$ (see \eqref{PadZ}) and $P\in\End(\ggo)$ is given by
$$
P=-R(Z),
$$
where $R$ denotes LSA right-multiplication (i.e. $R(X)Y=Y\cdot X$).
\end{remark}

\begin{proof}
It follows from \eqref{m-CRform} and \eqref{oc} that
\begin{align*}
\la P(X,0),(Y,0)\ra =& p((X,0),(0,-Y)) \\
=& \unm\la\theta(X)Y,A^*\ra-\unm\tr{\ad_\ggo{\theta(X)Y}}-\unm\tr{\theta(\theta(X)Y)} \\
=& \unm\la\theta(X)Y,A^*\ra-\unm\tr{\ad_\ggo{\theta(X)Y}} -\unm\sum\la\theta(\theta(X)Y)^te_i,e_i\ra \\
=& \unm\la\theta(X)Y,A^*\ra-\unm\tr{\ad_\ggo{\theta(X)Y}} -\unm\sum\la-[\theta(X)Y,e_i]_\ggo+\theta(e_i)^t\theta(X)Y,e_i\ra \\
=& \unm\la\theta(X)Y,A^*\ra-\unm\tr{\ad_\ggo{\theta(X)Y}}+\unm\tr{\ad_\ggo{\theta(X)Y}}-\unm\la\theta(X)Y,A\ra \\
=&\la\theta(X)Y,\unm(A^*-A)\ra = \la Y,\theta(X)^t\unm(A^*-A)\ra \\
=& \la Y,-[X,\unm(A^*-A)]_\ggo+\theta(\unm(A^*-A))^tX\ra \\
=& \la Y,\left(\unm\ad_\ggo{(A^*-A)}+\unm\theta((A^*-A))^t\right)X\ra,
\end{align*}
which proves formula \eqref{Psymp}.  The formula for $P\in\End(\ggo\oplus\ggo)$ follows from the fact that $P^{t_\omega}=P$, concluding the proof.
\end{proof}

\begin{remark}
It can be proved in much the same way that the Chern-Ricci operator of the hermitian structure $(\ggo\ltimes_{\theta^*}\ggo,\omega,g)$, which is the weak mirror image of $(\ggo\ltimes_\theta\ggo,\omega,g)$ (see Remark \ref{mirror}), is given by $P=\left[\begin{array}{cc} P&0\\ 0&P\end{array}\right]$, where $P=P^t\in\End(\ggo)$ is defined by
$$
\la PX,Y\ra=-\tr{\theta^*(\theta^*(X)Y)} = -\tr{L(X\cdot Y)}.
$$
\end{remark}

In the following computation of the Ricci curvature we are not assuming that $\omega$ is closed (i.e. condition \eqref{oc}).  The Ricci operator $\Ricci$ of $(\ggo\ltimes_\theta\ggo,g)$ can be computed by using for example \cite[Section 2.3]{homRF}, which gives
\begin{align}
\Ricci=& \left[\begin{array}{cc} \Ricci_\ggo-C_\theta-S(\ad_\ggo{H_\theta})&0\\ 0&\unm\sum [\theta(e_i),\theta(e_i)^t]-S(\theta(H)) \end{array}\right], \label{m-ric} \\\notag \\
=& \left[\begin{array}{cc} M_\ggo-\unm B_\ggo-C_\theta-S(\ad_\ggo{H})&0\\ 0&\unm\sum [\theta(e_i),\theta(e_i)^t]-S(\theta(H)) \end{array}\right], \notag
\end{align}
where $\Ricci_\ggo$ is the Ricci operator of $(\ggo,\ip)$, $C_\theta$ is the positive semi-definite operator given by
$$
\la C_\theta X,Y\ra=\tr{S(\theta(X))S(\theta(Y))},
$$
$S(E):=\unm(E+E^t)$ denotes the symmetric part of an operator $E$, $M_\ggo$ is defined by $\tr{M_\ggo E}=-\unc\la\delta_{\lb_\ggo}(E),\lb_\ggo\ra$ (see \eqref{delta}) and $B_\ggo$ is the Killing form of $\ggo$ relative to $\ip$ (i.e. $\tr{\ad_{\ggo}{X}\ad_{\ggo}{Y}}=\la B_\ggo X,Y\ra$).  Here $H\in\ggo$ is defined by $\la H,X\ra=\tr{\ad{X}}$, or equivalently,
$$
H:=H_\ggo+H_\theta, \qquad \la H_\ggo,X\ra=\tr{\ad_\ggo{X}}, \qquad \la H_\theta,X\ra=\tr{\theta(X)}.
$$
Thus the scalar curvature equals
\begin{align}
R=&R_\ggo-\sum\tr{S(\theta(e_i))^2} -\tr{\ad_\ggo{H_\theta}} -\tr{\theta(H)}, \label{scalar}\\
=& -\unc |\lb_\ggo|^2 -\unm\tr{B_{\ggo}} -\sum\tr{S(\theta(e_i))^2} -|H|^2, \notag
\end{align}
where $R_\ggo$ is the scalar curvature of $(\ggo,\ip)$.

Furthermore, the anti-J-invariant component of $\Ricci$ is therefore given by
\begin{equation}\label{m-ricac}
\Ricci^{ac}= \left[\begin{array}{cc} S&0\\ 0&-S\end{array}\right],
\end{equation}
where
$$
S=\unm\Ricci_\ggo-\unm C_\theta -\unm S(\ad_\ggo{H_\theta})
-\unc\sum [\theta(e_i),\theta(e_i)^t]+\unm S(\theta(H)).
$$
It is easy to check that $H=A$ when $\omega$ is closed, from which follows that the Chern scalar curvature is given by
$$
\tr{P}=\la A,A^*\ra-|A|^2,
$$
(recall from Remark \ref{PRZ} that $\tr{P}=2\tr{\ad{Z}}=2\la H,Z\ra=2\la A,\unm(A^*-A)\ra$) and consequently, $\tr{P}=0$ when $\ggo\ltimes_\theta\ggo$ is unimodular.

\subsection{SCF-solitons}\label{exmirror}
We first note that a simple way to obtain a SCF-soliton of the form $(\ggo\ltimes_\theta\ggo,\omega,g)$ is when both $P$ and $S$ are multiples of the identity (see Examples \ref{u(2)} and \ref{gl2nice} for an explicit application).  Indeed, if $P=qI$ and $S=rI$, $q,r\in\RR$, then
$$
\Ricci^{ac}=\left[\begin{array}{cc} rI&0\\ 0&-rI\end{array}\right] = rI + \left[\begin{array}{cc} 0&0\\ 0&-2rI\end{array}\right]\in\RR I + \Der(\ggo\ltimes_\theta\ggo),
$$
and thus the almost-K\"ahler structure $(\ggo\ltimes_\theta\ggo,\omega,g)$ is a (strongly algebraic) SCF-soliton with $c=q+r$ (see \eqref{strongly}).

We have seen in Section \ref{sec-LG} that given a symplectic Lie algebra $(\ggo,\omega)$, the set of all compatible metrics can be identified with the orbit $\Spe(\omega)\cdot\lb$.  In the case $(\ggo\ltimes_\theta\ggo,\omega)$, in order to explore the existence of SCF-solitons, we can vary the LSA structure by
\begin{equation}\label{var}
L_\vp(X):=\vp L(\vp^{-1}X)\vp^{-1}, \quad \left[\begin{matrix} \vp&0\\ 0&\vp^{-1}\end{matrix}\right]\in\Spe(\omega), \quad \vp\in\Gl(\ggo), \quad \vp^t=\vp.
\end{equation}
The corresponding Lie bracket $\lb_\vp$ on $\ggo\oplus\ggo$ defined in \eqref{m-lb} is therefore defined in terms of its components $(\lb_{\vp})_\ggo=\vp[\vp^{-1}\cdot, \vp^{-1}\cdot]_\ggo$ and $\theta_\vp(X)=\vp^{-1}\theta(\vp^{-1}X)\vp$.  Recall that $(\ggo\ltimes_\theta\ggo,\omega)$ and $(\ggo\ltimes_{\theta_\vp}\ggo,\omega)$ are isomorphic as symplectic Lie algebras (see Proposition \ref{LSA-iso}) and that if in addition $\vp\in\Or(\ggo,\ip)$ (i.e. $\vp^2=I$), then the almost-K\"ahler structures $(\ggo\ltimes_\theta\ggo,\omega,g)$ and $(\ggo\ltimes_{\theta_\vp}\ggo,\omega,g)$ are equivalent (see Proposition \ref{LSA-equiv}).

\begin{example}\label{u(2)}
We consider the Lie algebra $\ggo=\ug(2)$ with (orthonormal) basis
$$
e_1=\left[\begin{smallmatrix} i&0\\ 0&i\end{smallmatrix}\right], \quad e_2=\left[\begin{smallmatrix} 0&-1\\ 1&0\end{smallmatrix}\right], \quad e_3=\left[\begin{smallmatrix} i&0\\ 0&-i\end{smallmatrix}\right], \quad e_4=\left[\begin{smallmatrix} 0&i\\ i&0\end{smallmatrix}\right],
$$
and Lie bracket
$$
[e_2,e_3]=2e_4, \quad [e_2,e_4]=-2e_3, \quad [e_3,e_4]=2e_2.
$$
If we identify $\ggo$ with the quaternion numbers $\HH$ via $\{e_1=1,e_2=i,e_3=j,e_4=k\}$, then the (associative) product on $\HH$ is an LSA structure defining the above Lie bracket.  By considering the variation
$$
\vp_t=\left[\begin{smallmatrix} t&&&\\ &1&& \\ &&1&\\ &&&1\end{smallmatrix}\right], \qquad t>0,
$$
we obtain the following one-parameter family of LSA structures:
$$
L_t(e_1)=\tfrac{1}{t}I, \quad L_t(e_2)=\left[\begin{smallmatrix} 0&-t&&\\ 1/t&0&& \\ &&0&-1\\ &&1&0\end{smallmatrix}\right],
\quad L_t(e_3)=\left[\begin{smallmatrix} &&-t&0\\ &&0&1 \\ 1/t&0&&\\ 0&-1&&\end{smallmatrix}\right],
\quad L_t(e_4)=\left[\begin{smallmatrix} &&0&-t\\ &&-1&0 \\ 0&1&&\\ 1/t&0&&\end{smallmatrix}\right],
$$
which define the same Lie bracket as above.  The Chern-Ricci operator of the corresponding almost-K\"ahler structure $(\ggo\ltimes_{\theta_t}\ggo,\omega,g)$ is given, according to \eqref{Psymp}, by
$$
P_t=\tfrac{-5+3t^2}{2t^2}I, \qquad \mbox{i.e.}\quad p_t=\tfrac{-5+3t^2}{2t^2}\omega,
$$
as it is easy to see that $A=-\tfrac{4}{t}e_1$ and $A^*=(\tfrac{1}{t}-3t)e_1$.  It is also straightforward to obtain that
$$
\Ricci_{\ggo_t}=\left[\begin{smallmatrix} 0&&&\\ &2&& \\ &&2&\\ &&&2\end{smallmatrix}\right], \qquad C_{\theta_t}=\Diag\left(\tfrac{4}{t^2}, \tfrac{(1-t^2)^2}{2t^2}, \tfrac{(1-t^2)^2}{2t^2}, \tfrac{(1-t^2)^2}{2t^2}\right),
$$
$H_t=H_{\theta_t}=-\tfrac{4}{t}e_1$, $S(\theta_t(H_t))=\tfrac{4}{t^2}I$, and
$$
\sum [\theta_t(e_i),\theta_t(e_i)^t]=\Diag\left(\tfrac{3(1-t^4)}{t^2}, \tfrac{1-t^4}{t^2}, -\tfrac{1-t^4}{t^2}, -\tfrac{1-t^4}{t^2}\right).
$$
We now use formula \eqref{m-ric} to get
$$
\Ricci_t=\Diag\left(-\tfrac{4}{t^2},\tfrac{-1+6t^2-t^4}{2t^2},\tfrac{-1+6t^2-t^4}{2t^2}, \tfrac{-1+6t^2-t^4}{2t^2}, -\tfrac{5+3t^4}{2t^2}, \tfrac{-9+t^4}{2t^2},\tfrac{-9+t^4}{2t^2},\tfrac{-9+t^4}{2t^2}\right).
$$

\begin{remark}
It is worth pointing out that $(\ggo\ltimes_{\theta_t}\ggo,g)$ has negative Ricci curvature (i.e. $\Ricci_t<0$) if and only if $t^2<3-\sqrt{8}$.
\end{remark}

The anti-J-invariant part  of $\Ricci_t$ (see \eqref{m-ricac}) is therefore given by
$$
\Ricci_t^{ac}= \left[\begin{array}{cc} S_t&0\\ 0&-S_t\end{array}\right], \qquad S_t=\Diag\left(\tfrac{-3+3t^4}{4t^2},\tfrac{4+3t^2-t^4}{2t^2},  \tfrac{4+3t^2-t^4}{2t^2}, \tfrac{4+3t^2-t^4}{2t^2}\right).
$$
Thus $S_t$ is a multiple of the identity if and only if $t^2=\frac{11}{5}$.  More precisely, for $t_0=\sqrt{\tfrac{11}{5}}$, we obtain that
$$
P_{t_0}=\tfrac{4}{11}I, \qquad \Ricci_{t_0}^{ac}=\left[\begin{array}{cc} \frac{72}{55}I&0\\ 0&-\frac{72}{55}I\end{array}\right] = \tfrac{72}{55}I + \left[\begin{array}{cc} 0&0\\ 0&-\frac{144}{55}I\end{array}\right]\in\RR I + \Der(\ggo\ltimes_{\theta_{t_0}}\ggo).
$$
This implies that the almost-K\"ahler structure $(\ggo\ltimes_{\theta_{t_0}}\ggo,\omega,g)$ is a (strongly algebraic) SCF-soliton with $c=\tfrac{92}{55}>0$, that is, a shrinking SCF-soliton.  We note that this structure is not K\"ahler ($\Ricci_t^{ac}\ne 0$), the Ricci operator is given by
$$
\Ricci_{t_0}=\tfrac{1}{55}\Diag(-100,92,92,92,-244,-52,-52,-52)
$$
and the scalar curvature equals $R_{t_0}=-\tfrac{224}{55}$.  A family of SCF-solutions containing this soliton is studied in Example \ref{BF-exa}.

\begin{remark}
By using a standard computational program, we found out that this SCF-soliton is the only one (up to isometry) satisfying $S=rI$ among all variations of the form $\vp=\Diag(a,b,c,d)$.
\end{remark}
\end{example}

\begin{example}\label{gl2nice}
The usual matrix multiplication on $\ggo=\glg_2(\RR)$ gives rise to an LSA structure defining the usual Lie bracket, which in the (orthonormal) basis
$$
e_1=\left[\begin{smallmatrix} 1&0\\ 0&1\end{smallmatrix}\right], \quad e_2=\left[\begin{smallmatrix} 0&-1\\ 1&0\end{smallmatrix}\right], \quad e_3=\left[\begin{smallmatrix} 1&0\\ 0&-1\end{smallmatrix}\right], \quad e_4=\left[\begin{smallmatrix} 0&1\\ 1&0\end{smallmatrix}\right],
$$
is given by
$$
[e_2,e_3]=2e_4, \quad [e_2,e_4]=-2e_3, \quad [e_3,e_4]=-2e_2.
$$
If we consider the variation
$$
\vp_{s,t}=\left[\begin{smallmatrix} s&&&\\ &t&& \\ &&1&\\ &&&1\end{smallmatrix}\right], \qquad s,t>0,
$$
then the corresponding two-parameter family of LSA structures is defined by
$$
\begin{array}{lcl}
L_{s,t}(e_1)=\tfrac{1}{s}I, &\qquad& L_{s,t}(e_2)=\left[\begin{smallmatrix} 0&-s/t^2&&\\ 1/s&0&& \\ &&0&-1/t\\ &&1/t&0\end{smallmatrix}\right], \\ \\
L_{s,t}(e_3)=\left[\begin{smallmatrix} &&s&0\\ &&0&-t \\ 1/s&0&&\\ 0&-1/t&&\end{smallmatrix}\right], &\qquad&
L_{s,t}(e_4)=\left[\begin{smallmatrix} &&0&s\\ &&t&0 \\ 0&1/t&&\\ 1/s&0&&\end{smallmatrix}\right],
\end{array}
$$
and the Lie bracket on $\ggo$ changes to
$$
[e_2,e_3]_{s,t}=\tfrac{2}{t}e_4, \quad [e_2,e_4]_{s,t}=-\tfrac{2}{t}e_3, \quad [e_3,e_4]_{s,t}=-2te_2.
$$
By a straightforward computation one obtains that
$$
P_{s,t}=\left(-\tfrac{5}{2s^2}+\tfrac{1}{2t^2}-1\right)I, \quad S_{s,t}=\left[\begin{smallmatrix} -\tfrac{3}{4s^2}+\tfrac{s^2}{4t^4}+\tfrac{s^2}{2}&&&\\ &\tfrac{3t^2}{2}-\tfrac{s^2}{2t^4}+\tfrac{2}{s^2}&& \\ &&-\tfrac{3t^2}{2}-\tfrac{s^2}{2}+\tfrac{2}{s^2}-3&\\ &&&-\tfrac{3t^2}{2}-\tfrac{s^2}{2}+\tfrac{2}{s^2}-3\end{smallmatrix}\right].
$$
It follows that $S_{s,t}$ is a multiple of the identity if and only if
$$
s^2=\frac{6t^4}{1-t^2}, \qquad f(t):=-108t^8+36t^6-97t^4-22t^2+11=0,
$$
and since $f(0)=11$ and $f(1)=-180$, there exists $t_0\in(0,1)$ such that $f(t_0)=0$ ($t_0\sim 0.49$).  By setting $s_0:=\sqrt{\frac{6t_0^4}{1-t_0^2}}$ ($\sim 0.68$), we obtain the expanding (strongly algebraic) SCF-soliton $(\ggo\ltimes_{\theta_{s_0,t_0}}\ggo,\omega,g)$ with $c\sim -3.61$, $P=qI$ ($q\sim -4.24$) and $S=rI$ ($r\sim 0.63$).  We note that this SCF-soliton has negative Ricci curvature:
$$
\Ricci_{s_0,t_0}\sim \Diag(-8.46,-0.43,-9.95,-9.95,-9.73,-1.70,-11.21,-11.21).
$$
\end{example}

\subsection{Bracket flow}\label{LSA-BF}
In this section, in order to study the SCF-evolution of almost-K\"ahler structures of the form $(\ggo\ltimes_\theta\ggo,\omega,g)$, we consider the bracket flow \eqref{intro2} and use Theorem \ref{BF-thm}.  According to \eqref{m-lb}, the Lie bracket of $\ggo\ltimes_\theta\ggo$ is determined by $\lambda:=\lb_{\ggo}$ and $\theta$ and so any bracket flow solution $\mu=\mu(t)$ will be given by a pair
$$
\mu(t)=(\lambda(t),\theta(t)).
$$
By using that
$$
P+\Ricac=\left[\begin{array}{cc} P+S&0\\ 0&P^t-S\end{array}\right],
$$
it is easy to see that the bracket flow equation $\mu'=\delta_\mu(P+\Ricac)$ is equivalent to the system
$$
\left\{\begin{array}{l}
\lambda'=\delta_\lambda(P+S), \\ \\
\theta'(X)=\theta((P+S)X) +[\theta(X),P^t-S], \qquad\forall X\in\ggo.
\end{array}\right.
$$
It follows from Theorem \ref{BF-thm} that $\omega$ remains closed relative to $\mu(t)$, that is,
$$
\lambda(X,Y)=-\theta(X)^tY+\theta(Y)^tX, \qquad \forall t,
$$
from which follows that the bracket flow is equivalent to the single equation for $\theta$ given by
\begin{equation}\label{BF-LSA}
\theta'(X)=\theta((P+S)X) +[\theta(X),P^t-S], \qquad\forall X\in\ggo,
\end{equation}
where $\lambda$ is defined in terms of $\theta$ as above (recall that $P$ and $S$ depend on $\theta$ and $\lambda$).  Indeed, if $Q_1:=P+S$ and $Q_2:=P^t-S$, then $\lambda$ evolves by

\begin{align*}
\lambda'(X,Y) =& -\theta'(X)^t(Y)+\theta'(Y)^t(X) \\
=& -\theta(Q_1X)^tY - [Q_2^t,\theta(X)^t]Y +\theta(Q_1Y)^tX + [Q_2^t,\theta(Y)^t]X \\
=& \lambda(Q_1X,Y)+\lambda(X,Q_1Y)-Q_1\lambda(X,Y) \\
& -\theta(Y)^t(Q_1+Q_2^t)X + \theta(X)^t(Q_1+Q_2^t)Y + (Q_1+Q_2^t)\lambda(X,Y),
\end{align*}

and since $Q_1+Q_2^t=2P$ and $P=-R(Z)$ (see Remark \ref{PRZ}), the LSA condition yields

\begin{align*}
\lambda'(X,Y) =& \delta_\lambda(Q_1)(X,Y) + 2\left(Y\cdot PX-X\cdot PY+P(X\cdot Y-Y\cdot X)\right) \\
=& \delta_\lambda(Q_1)(X,Y) + 2\Big(-Y\cdot (X\cdot Z) +X\cdot (Y\cdot Z) -(X\cdot Y)\cdot Z +(Y\cdot X)\cdot Z\Big) \\
=& \delta_\lambda(Q_1)(X,Y).
\end{align*}

\begin{example}\label{BF-exa}
For $\ggo=\ug(2)$ as in Example \ref{u(2)}, consider the two-parameter family of almost-K\"ahler structures $(\ggo\ltimes_{\theta_{a,b}}\ggo,\omega,g)$, where
$$
\begin{array}{lcl}
\theta_{a,b}(e_1)=aI, &\qquad& \theta_{a,b}(e_2)=\left[\begin{smallmatrix} 0&a&&\\ -b^2/a&0&& \\ &&0&-b\\ &&b&0\end{smallmatrix}\right], \\ \\
\theta_{a,b}(e_3)=\left[\begin{smallmatrix} &&a&0\\ &&0&b \\ -b^2/a&0&&\\ 0&-b&&\end{smallmatrix}\right], &\qquad&
\theta_{a,b}(e_4)=\left[\begin{smallmatrix} &&0&a\\ &&-b&0 \\ 0&b&&\\ -b^2/a&0&&\end{smallmatrix}\right],
\end{array}
$$
and so the corresponding Lie bracket on $\ug(2)$ is given by
$$
\lambda_{a,b}(e_2,e_3)=2be_4, \qquad \lambda_{a,b}(e_2,e_4)=-2be_3, \qquad \lambda_{a,b}(e_3,e_4)=2be_2.
$$
We note that this family corresponds to the variation $\vp=\Diag(-1/a,1/b,1/b,1/b)$.

If we denote by $\Theta_{a,b}(X)$ the right-hand side of bracket flow equation \eqref{BF-LSA}, then it is easy to compute that
$$
\begin{array}{lcl}
\Theta_{a,b}(e_1)=\alpha I, &\qquad& \Theta_{a,b}(e_2)=\left[\begin{smallmatrix} 0&\alpha&&\\ \gamma&0&& \\ &&0&-\beta\\ &&\beta&0\end{smallmatrix}\right], \\ \\
\Theta_{a,b}(e_3)=\left[\begin{smallmatrix} &&\alpha&0\\ &&0&\beta \\ \gamma&0&&\\ 0&-\beta&&\end{smallmatrix}\right], &\qquad&
\Theta_{a,b}(e_4)=\left[\begin{smallmatrix} &&0&\alpha\\ &&-\beta&0 \\ 0&\beta&&\\ \gamma&0&&\end{smallmatrix}\right],
\end{array}
$$
where
$$
\alpha:=-\tfrac{13}{4}a^3+\tfrac{3}{2}ab^2+\tfrac{3}{4}b^4/a, \quad \beta:=-\unm a^2b+3b^3-\unm b^5/a^2, \quad \gamma:=-\tfrac{9}{4}ab^2-\tfrac{9}{2}b^4/a+\tfrac{7}{4}b^6/a^3.
$$
This implies that the family is invariant under the bracket flow if and only if $(-b^2/a)'=\gamma$ follows from $a'=\alpha$ and $b'=\beta$, which can be checked in a straightforward way.  The bracket flow on the family of almost-K\"ahler structures $(\ggo\ltimes_{\theta_{a,b}}\ggo,\omega,g)$ therefore becomes the following ODE system for $a=a(t)$, $b=b(t)$:
\begin{equation}\label{BF-exa-eq}
\left\{\begin{array}{l}
a'= -\tfrac{13}{4}a^3+\tfrac{3}{2}ab^2+\tfrac{3}{4}b^4/a, \\ \\
b'= -\unm a^2b+3b^3-\unm b^5/a^2.
\end{array}\right.
\end{equation}
We can assume, up to equivalence, that $a,b>0$.  Note that the shrinking SCF-soliton found in Example \ref{u(2)} belongs to the family; namely, it is contained in the straight line $b=\sqrt{\tfrac{11}{5}}a$, on which the equation becomes $a'=c a^3$ for $c=\tfrac{92}{55}$.  By a standard qualitative analysis, one can obtain the following information on these SCF-solutions:

\begin{itemize}
\item They all develop a finite-time singularity ($T_+<\infty$) and converge asymptotically to the SCF-soliton solution $\left(a(t),\sqrt{\tfrac{11}{5}}a(t)\right)$, $a(t)=(-2ct+1)^{-1/2}$, $t\in(-\infty,\tfrac{1}{2c})$.

\item They are all {\it ancient} solutions (i.e. $T_-=-\infty$).

\item For the solutions above the soliton (i.e. $b>\sqrt{\tfrac{11}{5}}a$), we have that the Chern scalar curvature $\tr{P}=(-20a^2+12b^2)$ is always positive, it comes from $+\infty$, attains a global minimum and then goes to $+\infty$, as $t\to T_+$.  On the other hand, the solutions below the soliton always increase $\tr{P}$ from $-\infty$ toward $+\infty$.

\item The scalar curvature $R=\tfrac{-43a^4+18a^2b^2-3b^4}{2a^2}$ is always negative and goes from $-\infty$ to $-\infty$, reaching a global maximum for any solution.
\end{itemize}

We now analyze the convergence behavior.  It is easy to see that
$$
\lim\limits_{t\to T_+} \tfrac{4}{\sqrt{11}}\frac{(a,b)}{\sqrt{a^2+b^2}} = \left(\sqrt{\tfrac{5}{11}},1\right),
$$
and thus pointed convergence of a subsequence $(G_{a_0,b_0},c_k\omega(t_k),c_kg(t_k))$ toward the SCF-soliton $(G_{\sqrt{\tfrac{5}{11}},1},\omega,g)$ follows for some $c_k>0$ (see Corollary \ref{conv}), for any starting almost-K\"ahler structure $(G_{a_0,b_0},\omega,g)$, $a_0,b_0>0$.

Concerning backward convergence, we have that if $b<\sqrt{\tfrac{11}{5}}a$, then
$$
\lim\limits_{t\to-\infty} \frac{(a,b)}{\sqrt{a^2+b^2}} = (1,0).
$$
It is easy to see that $\ggo\ltimes_{\theta_{1,0}}\ggo$ is a solvable Lie algebra with nilradical isomorphic to $\hg_7$, the $7$-dimensional Heisenberg algebra.  Moreover, $(\ggo\ltimes_{\theta_{1,0}}\ggo,\omega,g)$ is an expanding SCF-soliton with
$$
P_{1,0}=-\tfrac{5}{2}I, \qquad \Ricac_{1,0}=-\tfrac{3}{4}I+D,
$$
where $D:=\unc\Diag(0,11,11,11,6,-5,-5,-5)\in\Der(\ggo\ltimes_{\theta_{1,0}}\ggo)$, and negative Ricci curvature
$$
\Ricci_{1,0}=\unm\Diag(-8,-1,-1,-1,-5,-9,-9,-9).
$$
On the other hand, for $b>\sqrt{\tfrac{11}{5}}a$ we obtain,
$$
\lim\limits_{t\to-\infty} \frac{(a,b)}{b^2/a} = (0,1),
$$
and hence $\tfrac{a}{b^2}\theta_{a,b}\to\theta_\infty$, as $t\to-\infty$, where the only nonzero coefficients of $\theta_\infty$ are $\theta_\infty(e_2)e_1=-e_2$,  $\theta_\infty(e_3)e_1=-e_3$ and  $\theta_\infty(e_4)e_1=-e_4$.  This implies that $\ggo\ltimes_{\theta_\infty}\ggo$ is a nilpotent Lie algebra and  $(\ggo\ltimes_{\theta_\infty}\ggo,\omega,g)$ is an expanding SCF-soliton, which is equivalent to $(G_{\mu_A},\omega,g)$ as in Example \ref{ex-Anilp} with $C=I$.
\end{example}


\begin{thebibliography}{MMM}
\bibitem[AS]{AndSlm} {\sc A. Andrada, S. Salamon}, Complex product structures on Lie algebras, {\it Forum Math.} {\bf 17} (2005), 261-295.

\bibitem[A]{Arr} {\sc R. Arroyo}, The Ricci flow in a class of solvmanifolds, {\it Diff. Geom. Appl.} {\bf 31} (2013), 472-485.

\bibitem[AL]{ArrLfn} {\sc R. Arroyo, R. Lafuente}, Homogeneous Ricci solitons in low dimensions, {\it Int. Math. Res. Notices}, in press (arXiv).

\bibitem[Ba]{Bas} {\sc O. Baues}, Left-symmetric algebras for $\glg(n)$, {\it Transactions Amer. Math. Soc.} {\bf 351} (1999), N. 7, 2979-2996.

\bibitem[Bo]{Bck} {\sc C. Bock}, On Low-Dimensional Solvmanifolds, preprint 2009 (arXiv).

\bibitem[By]{Bym} {\sc N. Boyom}, Models for solvable symplectic Lie groups, {\it Indiana Univ. Math. J.} {\bf 42} (1993), 1149-1168.

\bibitem[Bu]{Brd} {\sc D. Burde}, Left-invariant affine structures on reductive Lie groups, {\it J. Algebra} {\bf 181} (1996), 884-902.

\bibitem[CLP]{mirror} {\sc R. Cleyton, J. Lauret, Y-S Poon}, Weak mirror symmetry of Lie algebras, {\it J. Symp. Geom.} {\bf 8} (2010), 37–55.

\bibitem[CM]{CnsMcr} {\sc S. Console, M. Macr\`\i}, Lattices, cohomology and models for six dimensional almost abelian solvmanifolds, preprint 2012 (arXiv).

\bibitem[D]{Dai} {\sc S. Dai}, A curvature flow unifying symplectic curvature flow and pluriclosed flow, preprint 2013 (arXiv).

\bibitem[F]{Frn} {\sc E. Fern\'andez-Culma}, Soliton almost K\"ahler structures on $6$-dimensional nilmanifolds for the symplectic curvature flow, {\it J. Geom. Anal.}, in press (arXiv).

\bibitem[H]{Hsg} {\sc K. Hasegawa}, Complex and K\"ahler structures on compact solvmanifolds, {\it J. Symp. Geom.} {\bf 3} (2005), 749–767.

\bibitem[Hu]{Hng} {\sc H. Huang}, Lattices and harmonic analysis on some $2$-step solvable Lie groups, {\it J. Lie Theory} {\bf 13} (2003), 77-89.

\bibitem[J1]{Jbl2}  {\sc M. Jablonski}, Detecting orbits along subvarieties via the moment map, {\it M\"unster J. Math.} {\bf 3} (2010), 67-88.

\bibitem[J2]{Jbl}  {\sc M. Jablonski}, Homogeneous Ricci solitons, {\it J. reine angew. Math.} {\bf 699} (2015) 159-182.

\bibitem[JP]{JblPtr}  {\sc  M. Jablonski, P. Petersen}, A step towards the Alekseevskii Conjecture, preprint 2014 (arXiv).

\bibitem[Lf]{Lfn} {\sc R. Lafuente}, Scalar curvature behavior of homogeneous Ricci flows, {\it J. Geom. Anal.}, in press (arXiv).

\bibitem[LL1]{homRS} {\sc R. Lafuente, J. Lauret}, On homogeneous Ricci solitons, {\it Quart. J. Math.} {\bf 65} (2014), 399-419.

\bibitem[LL2]{alek} {\sc R. Lafuente, J. Lauret}, Structure of homogeneous Ricci solitons and the Alekseevskii conjecture, {\it J. Diff. Geom.} {\bf 98} (2014) 315-347.

\bibitem[L1]{finding}  {\sc J. Lauret}, Finding Einstein solvmanifolds by a variational method,
{\it Math. Z.} {\bf 241} (2002), 83-99.

\bibitem[L2]{praga}  {\sc J. Lauret}, Minimal metrics on nilmanifolds, Diff. Geom. and its Appl.,
proc. Conf. prague September 2004 (2005), 77-94 (arXiv).

\bibitem[L3]{cruzchica} {\sc J. Lauret}, Einstein solvmanifolds and nilsolitons, {\it Contemp. Math.} {\bf 491} (2009), 1-35.

\bibitem[L4]{spacehm}  {\sc J. Lauret}, Convergence of homogeneous manifolds, {\it J. London Math. Soc.} {\bf 86} (2012), 701-727.

\bibitem[L5]{homRF} {\sc J. Lauret}, Ricci flow of homogeneous manifolds, {\it Math. Z.} {\bf 274} (2013), 373-403.

\bibitem[L6]{SCF}  {\sc J. Lauret}, Curvature flows for almost-hermitian Lie groups, {\it Transactions Amer. Math. Soc.}, in press (arXiv).

\bibitem[LR]{CRF}  {\sc J. Lauret, E. Rodr\'\i guez-Valencia}, On the Chern-Ricci flow and its solitons for Lie groups, {\it Math. Nachrichten}, in press (arXiv).

\bibitem[LeW]{LeWng} {\sc H-V Le, G. Wang}, Anti-complexified Ricci flow on compact symplectic manifolds, {\it J.
reine angew. Math.} {\bf 530} (2001), 17-31.

\bibitem[LM]{LchMdn} {\sc A. Lichnerowicz, A. Medina}, On Lie groups with left invariant symplectic or K\"ahlerian structure, {\it Lett. Math. Phys.} {\bf 16} (1988), 225-235.

\bibitem[Lt]{Ltt} {\sc J. Lott}, On the long-time behavior of type-III Ricci flow solutions, {\it Math. Annalen} {\bf 339} (2007), 627-666.

\bibitem[NN]{NklNkn} {\sc Y. Nikolayevsky, Yu.G. Nikonorov}, Solvable Lie groups of negative Ricci curvature, {\it Math. Z.}, in press (arXiv).

\bibitem[O]{Ovn} {\sc G. Ovando}, Four dimensional symplectic Lie algebras, {\it Beitr. Alg. Geom.} {\bf 47} (2006), 419-434.

\bibitem[PW]{PtrWyl}  {\sc P. Petersen, W. Wylie}, On gradient Ricci solitons with symmetry,  {\it Proc. Amer. Math. Soc.} {\bf 137} (2009), 2085-2092.

\bibitem[P]{Pk} {\sc J. Pook}, Homogeneous and locally homogeneous solutions to symplectic curvature flow, preprint 2012 (arXiv).

\bibitem[S]{Smt} {\sc D. J. Smith}, Stability of the almost hermitian curvature flow, preprint 2013 (arXiv).

\bibitem[ST1]{StrTn1} {\sc J. Streets, G. Tian}, Hermitian curvature flow, {\it J. Eur. Math. Soc.} {\bf 13} (2011), 601-634.

\bibitem[ST2]{StrTn2} {\sc J. Streets, G. Tian}, Symplectic curvature flow, {\it J. reine angew. Math.}, in press (arXiv).

\bibitem[ST3]{StrTn3} {\sc J. Streets, G. Tian}, Regularity results for pluriclosed flow, {\it Geom. Top.} {\bf 17} (2013), no. 4, 2389–2429.

\bibitem[TW]{TstWnk} {\sc V. Tosatti, B. Weinkove}, On the evolution of a hermitian metric by its Chern-Ricci form, {\it J. Diff. Geom.} {\bf 99} (2015), 125-163.

\bibitem[V]{Vzz2} {\sc L. Vezzoni}, A note on canonical Ricci forms on $2$-step nilmanifolds, {\it Proc. Amer. Math. Soc.} {\bf 141} (2013), 325-333.

\bibitem[Z]{Zhn} {\sc Z. Zhang}, Scalar curvature behavior for finite-time singularity of K\"ahler-Ricci
flow, {\it Michigan Math. J.} {\bf 59} (2010), 419-433.
\end{thebibliography}
\end{document}